\RequirePackage{fix-cm}
\documentclass[smallextended]{svjour3}       
\smartqed  
\usepackage{amsmath,amsfonts,amssymb}
\usepackage{cases}
\usepackage{subfigure}
\usepackage{tikz-cd}
\usepackage{IEEEtrantools}
\usepackage{graphicx}

\newcommand{\R}{{\mathbb R}}

\newcommand{\mc}[1]{\mathcal {#1}}
\newcommand{\dg}{{\dagger}}
\newcommand{\n}{{*_N}}
\newcommand{\1}{{*_L}}
\newcommand{\m}{{*_M}}

\newcommand{\f}{{*_R}}
\usepackage[centering, margin={1.06in, 0.7in}, includeheadfoot]{geometry}
\usepackage[active,new,old]{correct} 

\begin{document}

\title{ Weighted Moore-Penrose inverses of arbitrary-order tensors
}


\author{
Ratikanta Behera$^{1}$ \and
Sandip Maji$^2$ \and 
R. N. Mohapatra$^1$       
}

\authorrunning{R. Behera et al.} 

\institute{Ratikanta Behera \at
\email{ratikanta.behera@ucf.edu} \and
R. N. Mohapatra  \at  
 \email{Ram.Mohapatra@ucf.edu} 
\and  Sandip Maji \at 
  \email{majisandip.378@gmail.com}
  \newline\newline   
  \begin{enumerate}
      \item Department of Mathematics, University of Central Florida, Orlando, FL 32816, USA.
\newline  
\item Department of Mathematics and Statistics, Indian Institute of Science Education and Research Kolkata, \\West Bengal, 741242, India
  \end{enumerate}
}
\date{Received: date / Accepted: date}

\maketitle

\begin{abstract}
Within the field of multilinear algebra, inverses and generalized inverses of tensors based on the  Einstein product have been investigated over the past few years. In this paper, we explore the singular value decomposition and full-rank decomposition of arbitrary-order tensors using  {\it reshape} operation. Applying range and null space of tensors along with the reshape operation; we further study the Moore-Penrose inverse of tensors and their cancellation properties via the Einstein product. Then we discuss weighted Moore-Penrose inverses of arbitrary-order tensors using such product. Following a specific algebraic approach, a few characterizations and representations of these inverses are explored. In addition to this, we obtain a few necessary and sufficient conditions for the reverse-order law to hold for weighted Moore-Penrose inverses of arbitrary-order tensors.
\end{abstract}

\section{Introductions}
\subsection{Background and motivation}
\label{sec1}
Tensors or hypermatrix are multidimensional generalizations of vectors and matrices, and have attracted tremendous interest in recent years (see \cite{kolda,Marcar,loan,Qil05,Sho13}). Indeed, multilinear systems are closely related to tensors and such systems are encountered in a number of fields of practical interest, i.e., signal processing  (see \cite{LalmV00,SidLaP17,CopBol89}), scientific computing (see \cite{BeyMo05,ShiWeL13,BraliNT13}), data mining \cite{Chew07}, data compression and retrieval of large structured data (see  \cite{SilL08,MaoChLY18}). 
Further, the Moore-Penrose  inverse of tensors plays an important role in  solving such multilinear systems (see \cite{bm,JinBa17,ma2019}) and the reverse-order law for the Moore-Penrose inverses of tensors yields a class of interesting problems that are fundamental in the theory of generalized inverses of tensors (see \cite{PanRad18,JR_rev}). In view of these, multilinear algebra is drawing more and more attention from researchers (see \cite{JinBa17,Bader:2006,Marcar,Kru77,LalmV00}), specifically, the recent findings in (see \cite{bm,BraliNT13,weit2,PanRad18,psmv18,sun}), motivate us to study this subject in the framework of arbitrary-order tensors.

Let $\mathbb{C}^{I_1\times\cdots\times I_N} (\mathbb{R}^{I_1\times\cdots\times I_N})$  be the set of  order
$N$ and dimension $I_1 \times \cdots \times I_N$ tensors over the complex (real)
field $\mathbb{C}(\mathbb{R})$. Let $\mc{A} \in \mathbb{C}^{I_1\times\cdots\times I_N}$ be a multiway array with  $N$-th order tensor, and $I_1, I_2, \cdots, I_N$ be dimensions of the first, second,$\cdots$, $N$th way,
 respectively. Indeed, a matrix is a second order tensor, and a vector is a first order tensor. We  denote ${\R}^{m \times n}$ to be the  set of all ${m \times n}$ matrices with real entries. Note that throughout the paper, tensors are represented in calligraphic letters like  $\mc{A}$, and the notation $(\mc{A})_{i_1...i_N}= a_{i_1...i_N}$ represents the scalars. Each entry of $\mc{A}$ is denoted by $a_{i_1...i_N}$.
The Einstein product (see \cite{ein}) $ \mc{A}\n\mc{B} \in \mathbb{C}^{I_1\times\cdots\times
I_M \times K_1 \times\cdots\times K_L }$ of tensors $\mc{A} \in \mathbb{C}^{I_1\times\cdots\times I_M \times J_1
\times\cdots\times J_N }$ and $\mc{B} \in
\mathbb{C}^{J_1\times\cdots\times J_N \times K_1 \times\cdots\times
K_L }$   is defined
by the operation $\n$ via
\begin{equation}\label{Eins}
(\mc{A}\n\mc{B})_{i_1...i_M k_1...k_L}
=\displaystyle\sum_{j_1...j_N}a_{{i_1...i_M}{j_1...j_N}}b_{{j_1...j_N}{k_1...k_L}}.
\end{equation}
The Einstein product is not commutative but associative, and distributes with respect to tensor addition. Further, cancellation does not work but there is a multiplicative identity tensor $\mc{I}$. 
This type of product of tensors is  used in the study of the theory of relativity \cite{ein} and also used in the area of continuum mechanics \cite{lai}.

On the other hand, one of the most successful developments in the world of linear algebra is the concept of Singular Value Decomposition (SVD) of matrices \cite{BenGr74}. This concept gives us important information about a matrix such as its rank, an orthonormal basis for the column or row space, and reduction to a diagonal form \cite{Tian04}. Recently this concept is also used in low rank matrix approximations \cite{Gra04,ishteva2011best,ye2005generalized}. Since tensors are natural multidimensional generalizations of matrices, there are many applications involving arbitrary-order tensors. Further,  the problem of decomposing tensors is approached in a variety of ways by extending the SVD, and extensive studies have exposed many aspects of such decomposition and its applications ( see for example \cite{CheQZ17,kolda,Kru77,LalmV00,SidLaP17,LIANG2018}). However, the existing framework of SVD of tensors appears to be insufficient and/or inadequate in several situations. 

The aim of this paper is to present a proper generalization of the SVD of arbitrary-order tensors under Einstein tensor product. In fact, the existing form \cite{BraliNT13} of the SVD is well suited for square tensors, which is defined as follows:

\begin{definition}{(Definition 2.8, \cite{BraliNT13})}: 
The transformation defined as; \\
$f : \mathbb{T}_{I,J,I,J}(\mathbb{R}) \longrightarrow \mathbb{M}_{IJ,IJ}(\mathbb{R})$ with $f(\mc{A}) = A$ and defined component wise as;
\begin{equation}
\begin{tikzcd}
(\mc{A})_{ijij} \arrow[r, "f"] & (A)_{[i+(j-1)I][i+(j-1)I]}
\end{tikzcd}
\end{equation}
where $\mathbb{T}_{I,J,I,J}(R) = \{ \mc{A} \in \mc{R}^{I\times J \times I \times J}~:~det(f(\mc{A})) \neq 0 \}$. In general for any even order tensor, the transformation is defined as;
$f : \mathbb{T}_{I_1,...,I_N,I_1,...,J_N}(\mathbb{R}) \longrightarrow \mathbb{M}_{I_1...I_N,J_1...J_N}(\mathbb{R})$ 
\begin{equation}\label{11}
\begin{tikzcd}
  (\mc{A})_{i_1...i_N j_1...j_N} \arrow[r, "f"] & (A)_{\left[i_1+\sum_{k=2}^{N} ({i_k - 1}) \prod_{l=1}^{k-1} I_l][j_1+\sum_{k=2}^{N} (j_k-1) \prod_{l=1}^{k-1} J_l\right]}
\end{tikzcd}
\end{equation}
\end{definition}
Using the above Definition and Theorem 3.17 in \cite{BraliNT13}, we obtain the SVD of a tensor $\mc{A}\in \mathbb{R}^{I\times J \times I \times J}$, which can be extended only to any square tensor, i.e., for  $\mc{A}\in \mathbb{R}^{I_1\times I_2 \times \cdots  \times I_N \times I_1\times I_2 \times \cdots  \times I_N}$. 
Extension of the SVD for an arbitrary-order tensor using this method \cite{BraliNT13} is impossible, since $f$ is not a homomorphism for even-order and/or arbitrary-order tensors. In fact, the Einstein product is not defined for the following two even-order tensors,  $\mc{A}\in \mathbb{R}^{I_1\times I_2 \times J_1\times J_2} $ and $\mc{B}\in \mathbb{R}^{I_1\times I_2 \times J_1 \times J_2} $, i.e., $\mc{A}*_2\mc{B}$ is not defined. 
Therefore, our aim in this paper is to find the SVD for any arbitrary order tensors using {\it reshape} operation, which is discussed in the next section.

In addition, recently there has been increasing interest in analyzing inverses and generalized inverses of tensors based on different tensor products (see\cite{SahBe20,Wei18,JinBa17,BraliNT13,sun}). The representations and properties of the ordinary tensor inverse were introduced in \cite{BraliNT13}. This interpretation is extended to the Moore-Penrose inverse of tensors in \cite{sun} and investigated for a few characterizations of different generalized inverses of tensors via Einstein product in \cite{bm}. Appropriately, Behera and Mishra \cite{bm} posed the open question: {\it ``Does there exist a full rank decomposition of tensors ? If so, can this be used to compute the Moore-Penrose inverse of a tensor''?}  It is worth mentioning that Liang and Zheng \cite{LIANG2018} investigated this question and discuss the computation of Moore-Pensore inverse of tensors using full rank decomposition. Here we extend this work for weighted Moore-Pensore inverse using {\it reshape} operation. 
When investigating on the weighted Moore-Pensore inverse, we find some interesting characterization, which is discussed in the next Section.

Recently, Panigrahy and Mishra \cite{pami1} investigated the Moore-Penrose inverse of a product of two tensors via Einstein product. Using such theory of Einstein product, Stanimirovic, et al. \cite{psmv18} also introduced some basic properties of the range and null space of multidimensional arrays, and the effective definition of the tensor rank, termed as reshaping rank. In this respect, Panigrahy et al. \cite{PanRad18} obtained a few necessary and sufficient conditions for the reverse order law for the Moore-Penrose inverses of tensors, which can be used to simplify various tensor expressions that involve inverses of tensor products \cite{dingwei}. Since then, many authors investigate the reverse order law for various classes of generalized inverses of tensors \cite{chetheory,JR_rev,Mispa18}. At the same time, the representations of the weighted Moore-Penrose inverse \cite{weit2} of an even-order tensor was introduced via the Einstein product. In this context, we focus our attention on exploring some characterizations and representation of weighted Moore-Penrose inverses of arbitrary-order tensors.

In this paper, we study the weighted Moore-Penrose inverse  of an arbitrary-order tensor. This study can lead to the enhancement of the computation of SVD and full rank decomposition of arbitrary-order tensor using {\it reshape} operation. With that in mind, we discuss some identities involving the weighted Moore-Penrose inverses of tensors and then obtain a few necessary and sufficient conditions of the reverse order law for the weighted Moore-Penrose inverses of arbitrary-order tensors via the Einstein product.

\subsection{Outline}
We organize the paper as follows: In the next subsection, we introduce some notations and definitions which are helpful in proving the main results of this paper. In Section 2, we provide the main results of the paper. In order to do so, we introduce SVD and full rank decomposition of an arbitrary-order tensor using reshape operation. Within this framework, the Moore-Penrose and the generalized weighted Moore-Penrose inverse for arbitrary-order tensor is defined. Furthermore, we obtain several identities involving the weighted Moore-Penrose inverses of tensors via Einstein product. The Section 3 contains a few necessary and sufficient conditions of the reverse-order law for the weighted Moore-Penrose inverses of tensors.

\subsection{Notations and definitions}

For convenience, we first briefly explain a few essential facts about the Einstein product of tensors, which are found in \cite{bm,BraliNT13,sun}.  For a tensor $ \mc{A}=(a_{{i_1}...{i_M}{j_1}...{j_N}})
 \in \mathbb{C}^{I_1\times\cdots\times I_M \times J_1 \times\cdots\times J_N}$, the tensor $\mc{B}
  =(b_{{j_1}...{j_N}{i_1}...{i_M}}) \in
   \mathbb{C}^{J_1\times\cdots\times J_N
\times I_1 \times\cdots\times I_M}$ is said to be conjugate transpose of $\mc{A}$, if $ b_{{j_1}...{j_N}{i_1}...{i_M}} =\overline{a}_{{i_1}...{i_M}{j_1}...{j_N}}$ and $\mc{B}$ is denoted by $\mc{A}^*$. When $b_{{j_1}...{j_N}{i_1}...{i_M}} = {a}_{{i_1}...{i_M}{j_1}...{j_N}}$, $\mc{B}$
 is the {\it transpose} of $\mc{A}$, denoted by $\mc{A}^T$. The definition of a diagonal tensor was introduced as a square tensor in \cite{BraliNT13} (Definition 3.12) and  as an even-order tensor in \cite{sun}.  Here we introduce arbitrary-order diagonal tensor, as follows: 
\begin{definition} 
A tensor $ \mc{D} \in \mathbb{C}^{I_1\times \cdots \times I_M\times J_1 \times \cdots \times J_N} $  with entries $ d_{{i_1}...{i_M}{j_1}...{j_N}}$
  is
   called a {\it diagonal
   tensor} if 
   $d_{{i_1}...{i_M}{j_1}...{j_N}} = 0 $, when\\ $ {\left[i_1+\sum_{k=2}^{M} ({i_k - 1}) \prod_{l=1}^{k-1} I_l] \neq [j_1+\sum_{k=2}^{N} (j_k-1) \prod_{l=1}^{k-1} J_l\right]}$.
\end{definition}
Now we  recall the definition of an identity  tensor below.

\begin{definition} (Definition 3.13, \cite{BraliNT13}) A tensor  $ \mc{I}_N \in \mathbb{C}^{J_1\times \cdots \times J_N\times J_1 \times \cdots \times J_N} $ with entries 
     $ (\mc{I}_N)_{i_1i_2 \cdots i_Nj_1j_2\cdots j_N} = \prod_{k=1}^{N} \delta_{i_k j_k}$,
   where
\begin{numcases}
{\delta_{i_kj_k}=}
  1, &  $i_k = j_k$,\nonumber
  \\
  0, & $i_k \neq j_k $.\nonumber
\end{numcases}
 is  called a {\it  unit tensor or identity tensor}.
\end{definition}
 Note that throughout the paper, we denote $\mc{I}_M $,$ \mc{I}_L $ and $ \mc{I}_R $ as identity tensors in the space $ \mathbb{C}^{I_1\times \cdots \times I_M\times I_1 \times \cdots \times I_M} $, $ \mathbb{C}^{K_1\times \cdots \times K_L\times K_1 \times \cdots \times K_L} $ and $\mathbb{C}^{H_1\times \cdots \times H_R\times H_1 \times \cdots \times H_R} $ respectively.
 Further, a tensor $\mc{O}$ denotes the {\it zero tensor} if  all the entries are zero. 
 A tensor $\mc{A}\in
\mathbb{C}^{I_1\times\cdots\times I_N \times I_1 \times\cdots\times
I_N}$ is {\it Hermitian}  if  $\mc{A}=\mc{A}^*$ and {\it skew-Hermitian} if $\mc{A}= - \mc{A}^*$. Subsequently, a tensor $\mc{A}\in \mathbb{C}^{I_1\times\cdots\times I_N \times I_1
\times\cdots\times I_N}$  is {\it unitary}  if  $\mc{A}\n
\mc{A}^*=\mc{A}^*\n \mc{A}=\mc{I}_N $, and {\it idempotent}  if $\mc{A}
\n \mc{A}= \mc{A}.$ In the case of tensors of real entries,
Hermitian, skew-Hermitian and unitary tensors are called {\it symmetric} (see Definition 3.16,
\cite{BraliNT13}), {\it skew-symmetric} and {\it
orthogonal} (see Definition 3.15, \cite{BraliNT13}) tensors,  respectively. 
Next we present the definition of the {\it reshape} operation, which was introduced earlier in \cite{psmv18}. This is a more general way of rearranging the entries in a tensor (it is also a standard Matlab function), as follows:

\begin{definition}{(Definition 3.1, \cite{psmv18})}: 
The 1-1 and onto reshape map, rsh,  is defined as: $rsh :  \mathbb{C}^{I_1 \times \cdots \times I_M \times J_1 \times \cdots \times J_N} \longrightarrow \mathbb{C}^{I_1\cdots I_M \times J_1\cdots J_N}$ with
\begin{equation}
    rsh(\mc{A}) = A = reshape(\mc{A},I_1\cdots I_M,J_1\cdots J_N) 
\end{equation}
      where $ \mc{A} \in \mathbb{C}^{I_1 \times \cdots \times I_M \times J_1 \times \cdots \times J_N} $  and the matrix $ A \in  \mathbb{C}^{I_1\cdots I_M \times J_1\cdots J_N}$. Further, the inverse reshaping is the mapping defined as: $rsh^{-1} :  \mathbb{C}^{I_1\cdots I_M \times J_1\cdots J_N} \longrightarrow  \mathbb{C}^{I_1 \times \cdots \times I_M \times J_1 \times \cdots \times J_N}$ with 
\begin{equation}
rsh^{-1}(A) = \mc{A} = reshape(A,I_1,\cdots,I_M,J_1,\cdots,J_N).
 \end{equation}
 where  the matrix $ A \in  \mathbb{C}^{I_1\cdots I_M \times J_1\cdots J_N}$ and the tensor $\mc{A} \in \mathbb{C}^{I_1 \times \cdots \times I_M \times J_1 \times \cdots \times J_N}$. 
\end{definition}
Further, Lemma 3.2 in \cite{psmv18} defined the rank of a tensor, $\mc{A} $, denoted by $ rshrank(\mc{A}) $ as 
\begin{equation}\label{21}
    rshrank(\mc{A}) = rank(rsh(\mc{A})).  
\end{equation}
Continuing this research, Stanimirovic et. al.  in \cite{psmv18} discussed the homomorphism properties of the rsh function,  as follows,

\begin{lemma}(Lemma 3.1 \cite{psmv18}): 
Let $ \mc{A}\in \mathbb{C}^{I_1\times\cdots\times I_M\times J_1 \times \cdots \times J_N } $ and $ \mc{B} \in \mathbb{C}^{J_1 \times \cdots \times J_N \times K_1 \times \cdots \times K_L} $ be given tensors. Then 
\begin{equation}\label{66}
rsh(\mc{A}\n \mc{B}) = rsh(\mc{A})rsh(\mc{B}) = AB \in \mathbb{C}^{I_1\cdots I_M \times K_1\cdots K_L }
\end{equation}
where $ A = rsh(\mc{A}) \in \mathbb{C}^{I_1\cdots I_M \times J_1 \cdots J_N}, B = rsh(\mc{B}) \in \mathbb{C}^{J_1\cdots J_N \times K_1\cdots K_L} $.
  
\end{lemma}

An immediate consequence of the above Lemma is the following:
\begin{equation}\label{77}
\mc{A}\n\mc{B} = rsh^{-1}(AB), ~~i.e.,~~  rsh^{-1}(AB) = rsh^{-1}(A)\n rsh^{-1}(B).
\end{equation}

Existence of SVD of any square tensor is discussed in \cite{BraliNT13}. Using this framework, Jun and Wei \cite{weit2} defined Hermitian positive definite tensors, as follows: 
\begin{definition}(Definition 1, \cite{weit2})
For $\mc{P} \in \mathbb{C}^{I_1\times\cdots\times I_N \times I_1 \times\cdots\times I_N} $, if there exists a unitary tensor $ \mc{U} \in \mathbb{C}^{I_1 \times \cdots \times I_N \times I_1 \times \cdots \times I_N}  $ such that
\begin{equation}\label{2.16}
\mc{P} =\mc{U} *_N \mc{D} *_N \mc{U}^*,
\end{equation} 
where $ \mc{D} \in \mathbb{C}^{I_1 \times \cdots \times I_N \times I_1 \times \cdots \times I_N} $ is a diagonal tensor with positive diagonal entries, then $\mc{P}$ is said to be a Hermitian positive definite tensor. 
\end{definition}
Further, Jun and Wei \cite{weit2} defined the square root of a Hermitian positive definite tensor, $\mc{P}$, as follows:   
\begin{equation*}
\mc{P}^{1/2} = \mc{U}*_N \mc{D}^{1/2} *_N \mc{U}^*
\end{equation*}
where $\mc{D}^{1/2}$ is the diagonal tensor, which obtained from $\mc{D}$ by taking the square root of all its diagonal entries. Notice that $\mc{P}^{1/2}$ is always non-singular and its inverse is denoted by $\mc{P}^{-1/2}$.
  
we now recall the definition of the range and the null space of arbitrary order tensors.

\begin{definition}( Definition 2.1,\cite{psmv18}):
  The null space and the range space of a tensor $ \mc{A} \in \mathbb{C}^{I_1\times \cdots\times I_M \times J_1 \times \cdots \times J_N} $ {are defined as follows:}
  \begin{eqnarray*}
\mc{N}(\mc{A}) = \{\mc{X} : \mc{A}*_N\mc{X} = \mc{O} \in  \mathbb{C}^{I_1\times \cdots\times I_M}\},~ and ~ \mathfrak{R}(\mc{A}) = \{ \mc{A}*_N\mc{X} : \mc{X} \in \mathbb{C}^{J_1 \times \cdots \times J_N } \}
  \end{eqnarray*}
\end{definition}
It is easily seen that $\mc{N}(A)$ is a subspace of $\mathbb{C}^{J_1\times \cdots\times J_N}$ and $\mathfrak{R}(\mc{A})$ is a subspace of $\mathbb{C}^{I_1 \times \cdots \times I_M}$. 
In particular, $\mc{N}(\mc{A}) = \{ \mc{O} \}$ {\it if and only if} $\mc{A}$ is left invertiable via $*_M$ operation and  $\mathfrak{R}(\mc{A}) = \mathbb{C}^{I_1 \times \cdots \times I_M} $ {\it if and only if} $\mc{A}$ is right invertiable via $*_N$ operation.

\section{Main results}
Mathematical modelling of problems in science and engineering typically involves solving multilinear systems; this becomes particularly challenging for problems having an arbitrary-order tensor. However, the existing framework on Moore-Penrose inverses of arbitrary-order tensor appears to  be insufficient and/or inappropriate.  It is thus of interest to study the theory of Moore-Penrose inverse of an arbitrary-order tensor via the Einstein product.

\subsection{Moore-Penrose inverses}
One of the most widely used methods is the SVD to compute Moore-Penrose inverse. Here we present a generalization of the SVD via the Einstein product.

\begin{lemma}\label{SVDlemma}
Let  $ \mc{A}\in \mathbb{C}^{I_1\times\cdots\times I_M\times J_1 \times \cdots \times J_N } $ with $ rshrank(\mc{A}) = r $. Then the SVD for tensor $ \mc{A}$ has the form 
\begin{equation}
     \mc{A} = \mc{U}\m \mc{D}\n\mc{V}^*
\end{equation}
where $ \mc{U} \in \mathbb{C}^{I_1 \times \cdots \times I_M \times I_1 \times \cdots \times I_M} $ and $ \mc{V} \in \mathbb{C}^{J_1 \times \cdots \times J_N \times J_1 \times \cdots \times J_N} $ are unitary tensors, and $ \mc{D} \in \mathbb{C}^{I_1 \times \cdots \times I_M \times J_1 \times \cdots \times J_N } $ is a diagonal tensor, 
defined by

\begin{equation*}
(\mc{D})_{i_1 \cdots i_M j_1 \cdots j_N} = \left\{ \,
\begin{IEEEeqnarraybox}[][c]{l?s} \IEEEstrut
\sigma_{I} >0, & if ${I} = {J} \in {\{1,2,...,r \}}$, \\
      0, & otherwise
      \IEEEstrut
    \end{IEEEeqnarraybox}
  \right.
\end{equation*}
 where $ I = [i_1+\sum_{k=2}^{M} ({i_k - 1}) \prod_{l=1}^{k-1} I_l] $ and $ {J} = [j_1+\sum_{k=2}^{N} (j_k-1) \prod_{l=1}^{k-1} J_l] $.

\end{lemma}

\begin{proof}
Let $ A = rsh(\mc{A}) \in \mathbb{C}^{I_1\cdots I_M \times J_1 \cdots J_N}  $. In the context of the SVD of the matrix $A$, one can write 
 $ A = U D V^* $, where $ U \in \mathbb{C}^{I_1\cdots I_M \times I_1 \cdots I_M}$ and $ V \in \mathbb{C}^{J_1 \cdots J_N \times J_1 \cdots J_N} $ are unitary matrices and $ D \in \mathbb{C}^{I_1\cdots I_M \times J_1 \cdots J_N} $ is a  diagonal matrix with 
\begin{equation*}
(D)_{I,J}  = \left\{ \,
\begin{IEEEeqnarraybox}[][c]{l?s} \IEEEstrut
\sigma_{I} >0, & if ${I} = {J} \in {\{1,2,...,r \}}$, \\
      0, & otherwise
      \IEEEstrut
    \end{IEEEeqnarraybox}
  \right.
\end{equation*}
From relations \eqref{66} and \eqref{77}, we can write  
 \begin{eqnarray}\label{remarkSVD}
 \mc{A} &=& rsh^{-1}(A) = rsh^{-1}(U D V^*) \\\nonumber
 &=& rsh^{-1}(U)\m rsh^{-1}(D)\n rsh^{-1}(V^*) 
 = \mc{U}\m \mc{D}\n\mc{V}^* 
 \end{eqnarray}
where $\mc{U} = rsh^{-1}(U), \mc{V} = rsh^{-1}(V)$ and $\mc{D} = rsh^{-1}(D) $.
 Further, $\mc{U}\m\mc{U}^* = rsh^{-1}(U U^*) = rsh^{-1}(I) = \mc{I}_M $ and $ \mc{V}\n\mc{V}^* = rsh^{-1}(V V^*) = rsh^{-1}(I) = \mc{I}_N $ gives $\mc{A} = \mc{U}\m \mc{D}\n\mc{V}^* $, where $ \mc{U}$ and $\mc{V} $ are unitary tensors and $ \mc{D} $ diagonal tensor.
\end{proof}

\begin{remark}\label{remark11}
The authors of in \cite{LIANG2018} proved the Theorem 3.2 for a square tensor. Here we proved for an arbitrary-order tensor. Specifically, the decomposition is followed from the homomorphism property of {\it reshape} map.
\end{remark}

Continuing this study, we recall the definition of the Moore-Penrose inverse of tensors in $\mathbb{C}^{I_1\times\cdots\times I_M \times J_1 \times ... \times J_N}$ via the Einstein product, which was introduced in \cite{LIANG2018} for arbitrary-order. 

\begin{definition}\label{defmpi}
Let $\mc{A} \in \mathbb{C}^{I_1\times\cdots\times I_M \times J_1
\times ... \times J_N}$. The tensor $\mc{X} \in
\mathbb{C}^{J_1\times\cdots\times J_N \times I_1 \times\cdots\times
I_M} $ satisfying the following four tensor equations:
\begin{eqnarray*}
&&(1)~\mc{A}\n\mc{X}\m\mc{A} = \mc{A};\\
&&(2)~\mc{X}\m\mc{A}\n\mc{X} = \mc{X};\\
&&(3)~(\mc{A}\n\mc{X})^* = \mc{A}\n\mc{X};\\
&&(4)~(\mc{X}\m\mc{A})^* = \mc{X}\m\mc{A},
\end{eqnarray*}
is called  the \textbf{Moore-Penrose inverse} of $\mc{A}$, and is
denoted by $\mc{A}^{\dg}$.
\end{definition}

Similarly, to the proof of Theorem 3.2 in \cite{sun}, we have the existence and uniqueness of the Moore-Penrose inverse of an arbitrary-order tensor in $\mathbb{C}^{I_1\times\cdots\times I_M \times J_1 \times ... \times J_N}$, as follows. 


\begin{theorem}\label{MPIunique}
The Moore-Penrose inverse of an arbitrary-order tensor, $ \mc{A} \in \mathbb{C}^{I_1\times \cdots \times I_M \times J_1 \times \cdots \times J_N} $ exists and is unique.
\end{theorem}

By straightforward derivation, the following results can be obtained,  which also hold (Lemma 2.3 and Lemma 2.6) in \cite{bm} for even-order tensor.

\begin{lemma}\label{revA}
Let $\mc{A}\in \mathbb{C}^{I_1\times\cdots\times I_M \times J_1
\times\cdots\times J_N }$. Then
\begin{enumerate}
\item[(a)] $\mc{A}^* = \mc{A}^{\dg} \m \mc{A} \n \mc{A}^*=\mc{A}^* \m \mc{A} \n \mc{A}^{\dg};$
\item[(b)] $\mc{A} = \mc{A} \n \mc{A}^* \m (\mc{A}^*)^{\dg} = (\mc{A}^*)^{\dg} \n \mc{A}^* \m \mc{A};$
\item[(c)] $\mc{A}^{\dg} = ({\mc{A}^*}\m{\mc{A})^{\dg}}\n\mc{A}^* = {\mc{A}^*}\m(\mc{A}\n\mc{A}^*)^{\dg}.$
\end{enumerate}
 \end{lemma}

From \cite{psmv18}, we present the relation of range space of multidimensional arrays 
which will be used to prove next Lemma.

\begin{lemma}\label{lemma25}(Lemma2.2, \cite{psmv18})
Let $\mc{A}\in \mathbb{C}^{I_1\times\cdots\times I_M \times J_1
\times\cdots\times J_N }$, $\mc{B}\in \mathbb{C}^{I_1\times\cdots\times I_M \times K_1
\times\cdots\times K_L }$. Then $ \mathfrak{R}(\mc{B}) \subseteq \mathfrak{R}(\mc{A}) $  if and only if there exists  $ \mc{U}\in \mathbb{C}^{J_1 \times \cdots\times J_N \times K_1 \times \cdots \times K_L } $ such that $ \mc{B} = \mc{A}\n \mc{U} $.

\end{lemma}

We now discuss the important relation between range and Moore-Penrose inverse of an arbitrary order tensor, which are mostly used in various section of this paper.

\begin{lemma}{}\label{ID} Let $\mc{A}\in \mathbb{C}^{I_1\times\cdots\times I_M \times J_1
\times\cdots\times J_N }$ and $\mc{B}\in \mathbb{C}^{I_1\times\cdots\times I_M \times K_1
\times\cdots\times K_L }$. Then\\
(a) $ \mathfrak{R}(\mc{B}) \subseteq \mathfrak{R}(\mc{A}) \Leftrightarrow \mc{A}*_N\mc{A}^\dag *_M\mc{B} =\mc{B} $,\\
(b) $ \mathfrak{R}(\mc{A}) = \mathfrak{R}(\mc{B}) \Leftrightarrow \mc{A}*_N\mc{A}^\dag =\mc{B}*_L\mc{B}^\dag $,\\
(c) $ \mathfrak{R}(\mc{A}) = \mathfrak{R}[(\mc{A}^\dag)^*] $ and $ \mathfrak{R}(\mc{A}^*) = \mathfrak{R}(\mc{A}^\dag) $.
\end{lemma}

\begin{proof}
(a) Using the fact that $\mathfrak{R}(\mc{A}*_N \mc{U}) \subseteq  \mathfrak{R}(\mc{A})$ for two tensors $\mc{A}$ and $\mc{U}$ in appropriate order, one can conclude $\mathfrak{R}(\mc{B}) \subseteq \mathfrak{R}(\mc{A}) $ from $ \mc{A}*_N\mc{A}^\dag *_M\mc{B} = \mc{B}$. Applying Lemma \ref{lemma25}, we conclude $\mc{B} = \mc{A}*_N\mc{P}$ from $\mathfrak{R}(\mc{B}) \subseteq \mathfrak{R}(\mc{A})$,  where $\mc{P} \in \mathbb{C}^{J_1\times \cdots \times J_N \times K_1 \times \cdots \times K_L  }$. 
Hence, $\mc{A}*_N\mc{A}^\dag *_M \mc{B} = \mc{A}*_N\mc{A}^\dag *_M\mc{A}*_N\mc{P} =\mc{B}.$

(b) From (a), we have $\mathfrak{R}(\mc{A}) = \mathfrak{R}(\mc{B})$ if and only if $\mc{A}*_N\mc{A}^\dag*_M \mc{B} = \mc{B}$ and $\mc{B}*_L\mc{B}^\dag *_M\mc{A} = \mc{A} 
$ which implies $ \mc{B}^\dg = \mc{B}^\dag \m  \mc{A}\n\mc{A}^\dg $. Then 
$\mc{A}*_N\mc{A}^\dag = \mc{B}*_L\mc{B}^\dag *_M\mc{A}*_N\mc{A}^\dag 
= \mc{B}*_L\mc{B}^\dag$.

(c) Using the Lemma \ref{revA} [(b),(c)], one can conclude that $\mathfrak{R}(\mc{A}) \subseteq \mathfrak{R}[(\mc{A}^\dag)^*] $ and $\mathfrak{R}[(\mc{A}^\dag)^*] \subseteq \mathfrak{R}(\mc{A})$   respectively. This follows $\mathfrak{R}(\mc{A}) = \mathfrak{R}[(\mc{A}^\dag)^*]$. Further, replacing $ \mc{A}$ by $\mc{A}^*$ and using the fact $( \mc{A}^*)^\dag = (\mc{A}^\dag)^*$ we obtain $ \mathfrak{R}(\mc{A}^*) = \mathfrak{R}(\mc{A}^\dag)$.
\end{proof}

Using the fact that $\mathfrak{R}(\mc{A}*_N \mc{B}) \subseteq  \mathfrak{R}(\mc{A})$ for two tensor $\mc{A}$ and $\mc{B}$ and the Definition-\ref{defmpi}, we get,
\begin{eqnarray}\label{lemma37}
\mathfrak{R}(\mc{A}*_N\mc{B}*_L\mc{B}^\dag ) = \mathfrak{R}(\mc{A}*_N\mc{B}),
\end{eqnarray}
where $\mc{A}\in \mathbb{C}^{I_1\times\cdots\times I_M \times J_1
\times\cdots\times J_N }$ and $\mc{B}\in \mathbb{C}^{J_1\times\cdots\times J_N \times K_1
\times\cdots\times K_L }$. Now using the method as in the proof of Lemma \ref{ID}, one can prove the next Lemma.

\begin{lemma}{}\label{IDR1} Let $\mc{A}\in \mathbb{C}^{I_1\times\cdots\times I_M \times J_1
\times\cdots\times J_N }$ and $\mc{B}\in \mathbb{C}^{K_1\times\cdots\times K_L \times J_1
\times\cdots\times J_N }$. Then\\
(a) $\mathfrak{R}(\mc{B}^*) \subseteq \mathfrak{R}(\mc{A}^*) \Leftrightarrow \mc{B}*_N\mc{A}^\dag *_M\mc{A} =\mc{B}$,\\
(b) $\mathfrak{R}(\mc{A}^*) =\mathfrak{R}(\mc{B}^*) \Leftrightarrow \mc{A}^\dag *_M\mc{A} = \mc{B}^\dag *_L\mc{B},$\\
(c) 
$\mathfrak{R}(\mc{A}*_N\mc{B}^\dag) =\mathfrak{R}(\mc{A}*_N\mc{B}^*)$.
\end{lemma}

Consider $ \mc{A},~ \mc{B},~\mc{X} \in \mathbb{C}^{I_1 \times \cdots \times I_N \times I_1 \times \cdots \times I_N}$ and all are invertible, the following equation  
\begin{equation} \label{cand1}
\mc{B}\n(\mc{A}\n\mc{X}\n\mc{B})^{-1}\n\mc{A} = \mc{X}^{-1}  
\end{equation}
is called the cancellation property of product of tensors $(\mc{A},\mc{B},\mc{X})$. 
When
the ordinary inverse is replaced by generalized inverse with suitable order, this cancellation property is not true in general. In this context, we concentrate to characterize all triples $ (\mc{A},\mc{B},\mc{X}) $ which satisfy 
\begin{equation}\label{canra}
\mc{X}^\dg = \mc{B} \f (\mc{A}\m\mc{X}\n\mc{B})^\dg \1 \mc{A}
\end{equation}
where  $ \mc{X} \in \mathbb{C}^{I_1 \times \cdots \times I_M \times J_1 \times \cdots \times J_N} $, $ \mc{A} \in \mathbb{C}^{K_1 \times \cdots \times K_L \times I_1 \times \cdots \times I_M} $ and $ \mc{B} \in \mathbb{C}^{J_1 \times \cdots \times J_N \times H_1 \times \cdots \times H_R} $.  The first result obtained below deals with the necessary condition of this properties.

\begin{lemma}\label{CAN}
Let $ \mc{X} \in \mathbb{C}^{I_1 \times \cdots \times I_M \times J_1 \times \cdots \times J_N} $, $ \mc{A} \in \mathbb{C}^{K_1 \times \cdots \times K_L \times I_1 \times \cdots \times I_M} $ and $ \mc{B} \in \mathbb{C}^{J_1 \times \cdots \times J_N \times H_1 \times \cdots \times H_R } $.\\ 
If $ \mc{X}^\dag = \mc{B}*_R(\mc{A}*_M\mc{X}*_N\mc{B})^\dag *_L\mc{A}  $, then 
 ${ \mc{X} = \mc{A}^\dag *_L \mc{A}*_M\mc{X} } ~~and~~ {\mc{X} =\mc{X}*_N\mc{B}*_R\mc{B}^\dag}$.
\end{lemma}

\begin{proof}
Let, $ \mc{X}^\dag = \mc{B}*_R(\mc{A}*_M\mc{X}*_N\mc{B})^\dag *_L\mc{A} $. It is quite obvious that  $ \mc{X}^\dag = \mc{X}^\dag \m \mc{A}^\dag \1 \mc{A} = \mc{B} \f \mc{B}^\dag \n \mc{X}^\dag $.
Hence, from Lemma~\ref{ID}[(a),(c)] and Lemma~\ref{IDR1}(a), we have 
$ \mathfrak{R}[(\mc{X}^\dag)^*] \subseteq \mathfrak{R}(\mc{A}^*), $\\$ i.e, \mathfrak{R}(\mc{X}) \subseteq \mathfrak{R}(\mc{A}^\dag)$  and $ \mathfrak{R}(\mc{X}^\dag) \subseteq \mathfrak{R}(\mc{B}), i.e., \mathfrak{R}(\mc{X}^*) \subseteq \mathfrak{R}[(\mc{B}^\dag)^*]$.
Which implies $ \mc{X} = \mc{A}^\dag\1\mc{A}\m\mc{X} $ and $ \mc{X} = \mc{X}\n\mc{B}\f\mc{B}^\dag $.
%
%
\end{proof}
The following example shows that converse of the above theorem is not true in general.
\begin{example}
Consider tensors
$\mc{A} = (a_{ijkl})_{1 \leq i,j,k,l \leq 2}  \in \mathbb{R}^{ 2\times 2\times 2\times 2}$, $ \mc{B} = \mc{A}^* $  and  $ \mc{X} = (x_{ijkl})_{1 \leq i,j,k,l \leq 2}  \in \mathbb{R}^{2\times 2\times 2\times 2} $  such that
    
\begin{eqnarray*}
a_{ij11} =
    \begin{pmatrix}
    1 & -1\\
    0 & 0\\
    \end{pmatrix},
a_{ij21} =
    \begin{pmatrix}
    -1 & 0\\
    0 & 0\\
    \end{pmatrix},
a_{ij12} =
    \begin{pmatrix}
    0 & -1\\
    1 & 0\\
    \end{pmatrix},
a_{ij22} =
    \begin{pmatrix}
    1 & 0\\
    0 & -1\\
    \end{pmatrix}, ~    
\end{eqnarray*}
and 
\begin{eqnarray*}
x_{ij11} =
    \begin{pmatrix}
    1 & -1\\
    0 & 0\\
    \end{pmatrix},
x_{ij12} =
    \begin{pmatrix}
    0 & 1\\
    0 & 0\\
    \end{pmatrix},
x_{ij21} =
    \begin{pmatrix}
    0 & 0\\
    -1 & 0\\
    \end{pmatrix},
x_{ij22} =
    \begin{pmatrix}
    0 & 0\\
    1 & 0\\
    \end{pmatrix}, ~    
\end{eqnarray*}
Then, 
\begin{eqnarray*}
(\mc{A}^\dg)_{ij11} =
    \begin{pmatrix}
    0 & -1\\
    0 & 0\\
    \end{pmatrix},
(\mc{A}^\dg)_{ij21} =
    \begin{pmatrix}
    -1 & -1\\
    1 & 0\\
    \end{pmatrix},
(\mc{A}^\dg)_{ij12} =
    \begin{pmatrix}
    -1 & -1\\
    0 & 0\\
    \end{pmatrix},
(\mc{A}^\dg)_{ij22} =
    \begin{pmatrix}
    0 & -1\\
    0 & -1\\
    \end{pmatrix}.    
\end{eqnarray*}
Thus we have 
\begin{equation*}
\mc{A}^\dg *_2 \mc{A} *_2 \mc{X} = \mc{X} \textnormal{~~~and~~~} \mc{X} *_2\mc{B} *_2 \mc{B}^\dg = \mc{X},
\end{equation*} 
But \begin{equation*}
\mc{B}*_2(\mc{A} *_2 \mc{X} *_2\mc{B})^\dg *_2\mc{A} \neq \mc{X}^\dg,
\end{equation*} 
where 
\begin{eqnarray*}
(\mc{B}*_2(\mc{A} *_2 \mc{X} *_2\mc{B})^\dg *_2\mc{A})_{ij11} =
    \begin{pmatrix}
    1 & 1\\
    \frac{1}{2} &  \frac{1}{2}\\
    \end{pmatrix},
(\mc{B}*_2(\mc{A} *_2 \mc{X} *_2\mc{B})^\dg *_2\mc{A})_{ij21} =
    \begin{pmatrix}
    0 & 0\\
    -\frac{1}{2} & \frac{1}{2}\\
    \end{pmatrix},\\
(\mc{B}*_2(\mc{A} *_2 \mc{X} *_2\mc{B})^\dg *_2\mc{A})_{ij12} =
    \begin{pmatrix}
    0 & 1\\
    0 & 0\\
    \end{pmatrix},
(\mc{B}*_2(\mc{A} *_2 \mc{X} *_2\mc{B})^\dg *_2\mc{A})_{ij22} =
    \begin{pmatrix}
    0 & -1\\
    0 & 0\\
    \end{pmatrix}, ~    
\end{eqnarray*}
\begin{eqnarray*}
(\mc{X}^\dg)_{ij11} =
    \begin{pmatrix}
    1 & 1\\
    0 & 0\\
    \end{pmatrix},
(\mc{X}^\dg)_{ij21} =
    \begin{pmatrix}
    0 & 0\\
    -\frac{1}{2} & \frac{1}{2}\\
    \end{pmatrix},
(\mc{X}^\dg)_{ij12} =
    \begin{pmatrix}
    0 & 1\\
    0 & 0\\
    \end{pmatrix},
(\mc{X}^\dg)_{ij22} =
    \begin{pmatrix}
    0 & 0\\
    0 & 0\\
    \end{pmatrix}.    
\end{eqnarray*}
\end{example}
However, the converse of Lemma~\ref{CAN} holds under the assumption of additional condition which is stated below.

\begin{lemma}
Let $ \mc{X} \in \mathbb{C}^{I_1 \times \cdots \times I_M \times J_1 \times \cdots \times J_N }, \mc{A} \in \mathbb{C}^{K_1 \times \cdots \times K_L \times I_1 \times \cdots \times I_M}  $ and $ \mc{B} \in \mathbb{C}^{J_1 \times \cdots \times J_N \times H_1 \times \cdots \times H_R  } $.
If ${ \mc{X} = \mc{A}^\dag *_L \mc{A}*_M\mc{X} } { =\mc{X}*_N\mc{B}*_R\mc{B}^\dag}$  along with the condition 
$\mc{K} = \mc{A}^\dag \1 (\mc{A}\m\mc{X})\n(\mc{A}\m\mc{X})^\dag \1 \mc{A} $ and $\mc{L} = \mc{B}\f (\mc{X}\n \mc{B})^\dag \m(\mc{X}\n \mc{B})\f \mc{B}^\dag $ are Hermitian, Then $ \mc{X}^\dag = \mc{B}*_R(\mc{A}*_M\mc{X}*_N\mc{B})^\dag *_L\mc{A} $.
\end{lemma}

\begin{proof}

 Let $ \mc{W} = \mc{B}*_R(\mc{A}*_M\mc{X}*_N\mc{B})^\dag *_L\mc{A} $.\\
Now, $ \mc{X}*_N\mc{W}*_M\mc{X} \\
= (\mc{A}^\dag*_L\mc{A}*_M\mc{X}*_N\mc{B}*_R\mc{B}^\dag)*_N\mc{B}*_R(\mc{A}*_M\mc{X}*_N\mc{B})^\dag *_L\mc{A}*_M(\mc{A}^\dag *_L\mc{A}*_M\mc{X}*_N\mc{B}*_R\mc{B}^\dag)  $\\
= $ \mc{A}^\dag *_L[(\mc{A}*_M\mc{X}*_N\mc{B})*_R(\mc{A}*_M\mc{X}*_N\mc{B})^\dag *_L(\mc{A}*_M\mc{X}*_N\mc{B})]*_R\mc{B}^\dag \\
= \mc{A}^\dag *_L(\mc{A}*_M\mc{X}*_N\mc{B}) *_R\mc{B}^\dag 
= \mc{X} $.\\
Further, $\mc{W}*_M\mc{X}*_N\mc{W} = \mc{B}*_R(\mc{A}*_M\mc{X}*_N\mc{B})^\dag *_L\mc{A}*_M\mc{X}*_N\mc{B} *_R(\mc{A}*_M\mc{X}*_N\mc{B})^\dag*_L \mc{A} = \mc{W}  $.\\
Again $ \mc{K} = \mc{X}*_N\mc{W} $ and $ \mc{L}=\mc{W}*_M\mc{X} $ are Hermitian.
Hence $ \mc{W} = \mc{X}^\dag $.
\end{proof}

From Lemma~\ref{CAN} It is clear that, If $ \mc{X}^\dg = \mc{B}\f(\mc{A}\m\mc{X}\n\mc{B})^\dg\1\mc{A} $, Then $ \mathfrak{R}(\mc{A} \m \mc{X}) = \mathfrak{R}(\mc{A}\m\mc{X}\n\mc{B}) $, which implies that $ (\mc{A}\m\mc{X}\n\mc{B})\f(\mc{A}\m\mc{X}\n\mc{B})^\dg = \mc{A}\m\mc{X}\n(\mc{A}\m\mc{X})^\dg $. It is easy to verify that 
$ \mc{X}\n\mc{X}^\dg = \mc{K} $ and $ \mc{X}^\dg\m\mc{X}=  \mc{L}$,   and both $ \mc{K}$ and $\mc{L}$ both are Hermitian. Therefore, a necessary and sufficient condition for the cancellation law can be stated as:

\begin{theorem}\label{CAN1}
Let $ \mc{X} \in \mathbb{C}^{I_1 \times \cdots \times I_M \times J_1 \times \cdots \times J_N }, \mc{A} \in \mathbb{C}^{K_1 \times \cdots \times K_L \times I_1 \times \cdots \times I_M}  $ and $ \mc{B} \in \mathbb{C}^{J_1 \times \cdots \times J_N \times H_1 \times \cdots \times H_R  } $.\\
$ \mc{X}^\dag = \mc{B}*_R(\mc{A}*_M\mc{X}*_N\mc{B})^\dag *_L\mc{A} $ if and only if
 ${ \mc{X} = \mc{A}^\dag *_L \mc{A}*_M\mc{X} } { =\mc{X}*_N\mc{B}*_R\mc{B}^\dag}$  and both $\mc{K} = \mc{A}^\dag \1 (\mc{A}\m\mc{X})\n(\mc{A}\m\mc{X})^\dag \1 \mc{A} $ and $\mc{L} = \mc{B}\f (\mc{X}\n \mc{B})^\dag \m(\mc{X}\n \mc{B})\f \mc{B}^\dag $ are Hermitian. 
\end{theorem}
We now proceed to discuss a few necessary and sufficient conditions for the  cancellation law.
\begin{corollary}\label{can1}
Let $ \mc{X} \in \mathbb{C}^{I_1 \times \cdots \times I_M \times J_1 \times \cdots \times J_N }, \mc{A} \in \mathbb{C}^{K_1 \times \cdots \times K_L \times I_1 \times \cdots \times I_M}  $ and $ \mc{B} \in \mathbb{C}^{J_1 \times \cdots \times J_N \times H_1 \times \cdots \times H_R }$,  and 
$ \mc{X}^\dag = \mc{B}*_R(\mc{A}*_M\mc{X}*_N\mc{B})^\dag *_L\mc{A} $ if and only if both the equations
\begin{equation}
  {  \mc{X}^\dag = (\mc{A}\m\mc{X})^\dag \1 \mc{A} ~~and~~ \mc{X}^\dag = \mc{B}\f(\mc{X}\n\mc{B})^\dag} ~~ are~ satisfied. 
\end{equation}

\end{corollary}

\begin{proof}
By taking  $ \mc{B} = \mc{I}  $ in Theorem~\ref{CAN1}, we have $ \mc{X}^\dag = (\mc{A}*_M\mc{X})^\dag *_L\mc{A}  $ if and only if  
$ \mc{A}^\dag*_L(\mc{A}\m\mc{X})\n(\mc{A}\m\mc{X})^\dg *_L\mc{A} $ is Hermitian and $ \mc{X}= \mc{A}^\dag *_L\mc{A}*_M\mc{X} $. Similarly with the special case $\mc{A} = \mc{I} $ in Theorem~\ref{CAN1}, we get 
$ \mc{X}^\dag =\mc{B}*_R(\mc{X}*_N\mc{B})^\dag $   if and  only if $ \mc{B}\f (\mc{X}\n \mc{B})^\dag \m(\mc{X}\n \mc{B})\f \mc{B}^\dag  $ is Hermitian and $ \mc{X} = \mc{X}*_N\mc{B}*_R\mc{B}^\dag  $. Using the fact of Theorem\ref{CAN1} one can prove the required result.
\end{proof}

Using Lemma~\ref{ID} and Lemma~\ref{IDR1} in the corollary~\ref{can1} one obtain the following result.

\begin{theorem}
Let $ \mc{X} \in \mathbb{C}^{I_1 \times \cdots \times I_M \times J_1 \times \cdots \times J_N }$, $ \mc{A} \in \mathbb{C}^{K_1 \times \cdots \times K_L \times I_1 \times \cdots \times I_M}  $ and $ \mc{B} \in \mathbb{C}^{J_1 \times \cdots \times J_N \times H_1 \times \cdots \times H_R}$, then
\begin{equation*}
\mc{X}^\dag = \mc{B}*_R(\mc{A}*_M\mc{X}*_N\mc{B})^\dag *_L\mc{A}
\end{equation*}
if and only if
\begin{equation*}
(\mc{A}*_M\mc{X})^\dag = \mc{X}^\dag *_M\mc{A}^\dag,~~  (\mc{X}*_N\mc{B})^\dag = \mc{B}^\dag *_N\mc{X}^\dag, ~~\mc{X} = \mc{A}^\dag *_L\mc{A}*_M\mc{X} ~~\textnormal{and} ~~\mc{X} = \mc{X}*_N\mc{B}*_R\mc{B}^\dag.
\end{equation*}
\end{theorem}

\begin{proof}
Suppose that $ \mc{X}^\dag = \mc{B}*_R(\mc{A}*_M\mc{X}*_N\mc{B})^\dag *_L\mc{A} $.
Then from Corollary~\ref{can1}, \\
$\mc{X}^\dag = (\mc{A}\m\mc{X})^\dag \1\mc{A} $ and  $ \mc{X}^\dag = \mc{B}\f(\mc{X}\n\mc{B})^\dag$.\\
Now, $ \mathfrak{R}(\mc{X}) = \mathfrak{R}[(\mc{X}^\dag)^*] = \mathfrak{R}[\mc{A}^* *_L\{(\mc{A}*_M\mc{X})^\dag\}^*] = \mathfrak{R}(\mc{A}^* *_L\mc{A}*_M\mc{X})$  and $ \mathfrak{R}(\mc{X}^*) = \mathfrak{R}(\mc{X}^\dag) = \mathfrak{R}[\mc{B}*_R(\mc{X}*_N\mc{B})^\dag] = \mathfrak{R}[\mc{B}*_R(\mc{X}*_N\mc{B})^*] = \mathfrak{R}(\mc{B}*_R\mc{B}^* *_N\mc{X}^*) $.\\
 Therefore, $ \mathfrak{R}(\mc{X}*_N\mc{X}^* *_M\mc{A}^*) \subseteq  \mathfrak{R}(\mc{X}) =  \mathfrak{R}(\mc{A}^* *_L\mc{A}*_M\mc{X}) \subseteq \mathfrak{R}(\mc{A}^*) $, i.e., $ \mathfrak{R}(\mc{X}*_N\mc{X}^* *_M\mc{A}^*) \subseteq \mathfrak{R}(\mc{A}^*) $
 and 
 $ \mathfrak{R}(\mc{A}^*\1\mc{A}\m\mc{X}) \subseteq \mathfrak{R}(\mc{X}) $  implies that $ (\mc{A}*_M\mc{X})^\dag = \mc{X}^\dag *_M\mc{A}^\dag $.\\
 $ \mathfrak{R}(\mc{X}) \subseteq \mathfrak{R}(\mc{A}^\dag) $  implies $ \mc{A}^\dag *_L\mc{A}*_M\mc{X} = \mc{X} $.
 Similarly from $ \mathfrak{R}(\mc{X}^*) = \mathfrak{R}(\mc{B}\f\mc{B}^*\n\mc{X}^*)  $ it follows that $ (\mc{X}*_N\mc{B})^\dag = \mc{B}^\dag *_N\mc{X}^\dag $,
 $ \mc{X} = \mc{X}*_N\mc{B}*_R\mc{B}^\dag $.\\
  Conversely, Using Lemma \ref{ID}(c), Lemma\ref{IDR1}(a) in the fact \\
  $ \mathfrak{R}[(\mc{X}^\dg)^*] = \mathfrak{R}(\mc{X}) \subseteq \mathfrak{R}(\mc{A}^\dg) = \mathfrak{R}(\mc{A}^*) $ and $ \mathfrak{R}(\mc{X}^\dg) = \mathfrak{R}(\mc{X}^*) \subseteq \mathfrak{R}[(\mc{B}^\dg)^*] = \mathfrak{R}(\mc{B}).$\\
One have $ \mc{X}^\dg = \mc{X}^\dg\m\mc{A}^\dg\1\mc{A} $ and $ \mc{X}^\dg = \mc{B}\f\mc{B}^\dg\n\mc{X}^\dg  $.\\ 
 Now, $ (\mc{A}*_M\mc{X})^\dag*_L\mc{A} = \mc{X}^\dag *_M\mc{A}^\dag *_L\mc{A} = \mc{X}^\dag  $ and $ \mc{B}*_R(\mc{X}*_N\mc{B})^\dag = \mc{B}*_R\mc{B}^\dag *_N\mc{X}^\dag = \mc{X}^\dg$. then by Corollary\ref{can1} proof is done. 
 \end{proof}


From the above theorem one can conclude the necessary and sufficient condition for cancellation law in terms of range.
\begin{lemma}\label{can2}
Let $ \mc{X} \in \mathbb{C}^{I_1 \times \cdots \times I_M \times J_1 \times \cdots \times J_N} $, $ \mc{A} \in \mathbb{C}^{K_1 \times \cdots \times K_L \times I_1 \times \cdots \times I_M} $ and $ \mc{B} \in \mathbb{C}^{J_1 \times \cdots \times J_N \times H_1 \times \cdots \times H_R   } $.
Then 
\begin{equation*}
\mc{X}^\dg = \mc{B}\f(\mc{A}\m\mc{X}\n\mc{B})^\dg\1\mc{A}
\end{equation*}
if and only if 
\begin{equation*}
\mathfrak{R}(\mc{X}) = \mathfrak{R}(\mc{A}^*\1\mc{A}\m\mc{X})   \textnormal{~~and~~}
\mathfrak{R}(\mc{X}^*) = \mathfrak{R}(\mc{B}\f\mc{B}^*\n\mc{X}^*).
\end{equation*}
\end{lemma}

\subsection{Weighted Moore-Penrose inverse}

Weighted Moore-Penrose inverse of even-order tensor, $\mc{A} \in \mathbb{C}^{I_1\times\cdots\times I_K \times J_1 \times\cdots\times J_K}$ was introduced in \cite{weit2}, very recently. Here we have discussed weighted Moore-Penrose inverse for an arbitrary-order tensor via Einstein product, which is a special case of generalized weighted Moore-Penrose inverse.

\begin{definition}\label{43}
Let $\mc{A} \in \mathbb{C}^{I_1\times\cdots\times I_M \times J_1 \times\cdots\times J_N} $,
 and a pair of invertible and Hermitian tensors $\mc{M} \in \mathbb{C}^{I_1\times\cdots\times I_M \times I_1 \times\cdots\times I_M}$ and $\mc{N} \in \mathbb{C}^{J_1\times\cdots\times J_N \times J_1 \times\cdots\times J_N}$. A tensor $\mc{Y} \in
 \mathbb{C}^{J_1\times\cdots\times J_N \times I_1 \times\cdots\times I_M}$ is said to be the \textbf{generalized weighted Moore-Penrose inverse} of $\mc{A}$ with respect to $\mc{M}$ and $\mc{N}$, if $\mc{Y}$ satisfies the following four tensor equations
\begin{eqnarray*}
&&(1)~\mc{A}*_N\mc{X}*_M\mc{A}= \mc{A};\\
&&(2)~\mc{X}*_M\mc{A}*_N\mc{X}= \mc{X};\\
&&(3)~(\mc{M}*_M \mc{A}*_N\mc{X})^* = \mc{M}*_M\mc{A}*_N\mc{X};\\
&&(4)~(\mc{N}*_N \mc{X}*_M\mc{A})^* =\mc{N}*_N \mc{X}*_M\mc{A}.
\end{eqnarray*}
In particular, when both $\mc{M},~\mc{N}$ are Hermitian positive definite tensors, the tensor  $\mc{Y}$ 
is called the {\textbf{weighted Moore-Penrose inverse}} of $\mc{A}$ and denote by $\mc{A}_{\mc{M},\mc{N}}^\dg$. 
\end{definition}
 However, the generalized weighted Moore-Penrose inverse $\mc{Y}$ does not always exist for any tensor $\mc{A}$, 
as shown below with an example. 
\begin{example}
Consider tensors
$~\mc{A}=(a_{ijk})
 \in \mathbb{R}^{\overline{2\times3}\times\overline{2}}$ and $\mc{M}=(a_{ijkl})
 \in \mathbb{R}^{\overline{2\times3}\times\overline{2\times3}}$ with $\mc{N}=(n_{ij})
 \in \mathbb{R}^{\overline{2}\times\overline{2}}$ such that
\begin{eqnarray*}
a_{ij1} =
    \begin{pmatrix}
    1 & 0 & 1 \\
   -1& 2 &  1
    \end{pmatrix},
a_{ij2} =
    \begin{pmatrix}
     2 & 0 & 3\\
     2 & 0 & 1
    \end{pmatrix}
     ~~\textnormal{and}~~~  N =
    \begin{pmatrix}
     2 & 0\\
   0 & -1
    \end{pmatrix} 
\end{eqnarray*}
with
\begin{eqnarray*}
m_{ij11} = 
    \begin{pmatrix}
2 & 0 & 0\\
0 & 0 & 0
    \end{pmatrix},
m_{ij12} =
    \begin{pmatrix}
0 & 2 & 0\\
0 & 0 & 0
    \end{pmatrix},
m_{ij13} =
    \begin{pmatrix}
0 & 0 & 1\\
0 & 0 & 0
    \end{pmatrix},
\end{eqnarray*}
\begin{eqnarray*}
m_{ij21} = 
    \begin{pmatrix}
0 & 0 & 0\\
-1 & 0 & 0
    \end{pmatrix},
m_{ij22} =
    \begin{pmatrix}
0 & 0 & 0\\
0 & 1 & 0
    \end{pmatrix},
m_{ij23} =
    \begin{pmatrix}
0 & 0 & 0\\
0 & 0 & 3
    \end{pmatrix},
\end{eqnarray*}
Then we have 
\begin{equation*}
\mc{A}^T *_2 \mc{M}*_2\mc{A}=\begin{pmatrix}
9 & 12\\
12 & 16 
    \end{pmatrix},
\end{equation*}
This shows $\mc{A}^T*_2\mc{M}*_2\mc{A}$ is not invertible. Consider the generalized weighted Moore-Penrose inverse $\mc{Y} \in \mathbb{R}^{\overline{2}\times \overline{2\times 3}}$ of the given tensor $\mc{A}$ is exist, then using relation (1) and relation (3) of Definition \ref{43}, we have
\begin{equation}\label{aas}
\mc{A}*_1 \mc{Y}*_2 \mc{M}^{-1}*_2 \mc{Y}^T *_1
\mc{A}^T *_2 \mc{M}*_2 \mc{A} = \mc{A}.
\end{equation}  
Since $(\mc{A}^T*_2\mc{A})^{-1}*_1\mc{A}^T*_2\mc{A} = \mc{I}$, then $\mc{A}$ is left cancellable, now \eqref{aas} becomes
 \begin{equation}
\mc{Y}*_2 \mc{M}^{-1}*_2 \mc{Y}^T *_1
\mc{A}^T *_2 \mc{M}*_2 \mc{A} = \mc{I},
\end{equation}  
this follows that $\mc{A}^T *_2 \mc{M}*_2 \mc{A}$ is invertiable, which is a contradiction.
\end{example}

At this point one may be interested to know when does the generalized weighted Moore-Penrose inverse exist ? The answer to this question is explained in the following theorem. 
\begin{theorem}\label{GWMPI}
Let $\mc{A} \in \mathbb{C}^{I_1\times\cdots\times I_M \times J_1 \times\cdots\times J_N}$. If both $\mc{M} \in \mathbb{C}^{I_1\times\cdots\times I_M \times I_1 \times\cdots\times I_M}$ and $\mc{N} \in \mathbb{C}^{J_1\times\cdots\times J_N \times J_1 \times\cdots\times J_N}$ are Hermitian  positive definite tensors. Then generalized weighted Moore-Penrose inverse of an arbitrary-order tensor $\mc{A}$ exists and is unique, i.e., there exist a unique tensor $\mc{X} \in
 \mathbb{C}^{J_1\times\cdots\times J_N \times I_1 \times\cdots\times I_M}$, such that,  
\begin{equation} \label{2.17}
\mc{X}=\mc{A}_{\mc{M},\mc{N}}^\dg = \mc{N}^{-1/2} *_N ({\mc{M}^{1/2} *_M \mc{A} *_N\mc{N}^{-1/2}})^\dg *_M \mc{M}^{1/2}
\end{equation}
where $\mc{M}^{1/2}$ and $\mc{N}^{1/2}$ are square roots of $\mc{M}$ and $\mc{N}$ respectively, satisfy all four relations of Definition \ref{43}. 
\end{theorem}

One can prove the above theorem, using  Theorem 1 in \cite{weit2} and Theorem \ref{MPIunique}. Further, it is know that identity tensors are always Hermitian and positive definite, therefore, for any 
$ \mc{A} \in \mathbb{C}^{I_1 \times \cdots \times I_M \times J_1 \times \cdots \times J_N} $, $ \mc{A}^\dg_{\mc{I}_M,\mc{I}_N} $ exists and $ \mc{A}^\dg_{\mc{I}_M,\mc{I}_N} = \mc{A}^\dg $, which is called the Moore-Penrose inverse of $\mc{A}$. Specifically, if we take $ \mc{M} =  \mc{I}_M $ or $\mc{N} = \mc{I}_N $ in Eq.\eqref{2.17}, then the following identities are hold.

\begin{corollary}\label{3.4}
Let $ \mc{A} \in \mathbb{C}^{I_1 \times \cdots \times I_M \times J_1 \times \cdots \times J_N} $. Then 
\begin{eqnarray*}\label{3.3.1}
&&(a)~\mc{A}_{\mc{M},\mc{I}_N}^\dg  = (\mc{M}^{1/2} *_M \mc{A})^\dg *_M \mc{M}^{1/2},\\
\label{3.3.2}
&&(b)~\mc{A}_{\mc{I}_M,\mc{N}}^\dg  = \mc{N}^{-1/2} *_N (\mc{A} *_N \mc{N}^{-1/2})^\dg.
\end{eqnarray*}
\end{corollary}

Using the Definition \ref{43} and following the Lemma~2 in \cite{weit2} one can write  
 $ (\mc{A}_{\mc{M},\mc{N}}^\dg)_{\mc{N},\mc{M}}^\dg = \mc{A} $ and $ (\mc{A}_{\mc{M},\mc{N}}^\dg)^* = (\mc{A}^*)^\dg_{\mc{N}^{-1}, \mc{M}^{-1}} $, where $\mc{A}$ is any arbitrary-order tensor.

Now we define weighted conjugate transpose of a arbitrary-order tensor, as follows. 
\begin{definition}\label{210}
  Let $ \mc{M} \in \mathbb{C}^{I_1\times\cdots\times I_M \times I_1 \times\cdots\times I_M}$ and $\mc{N} \in \mathbb{C}^{J_1\times\cdots\times J_N \times J_1 \times\cdots\times J_N} $  are Hermitian positive definite tensors, the \textbf{weighted conjugate transpose} of $ \mc{A} \in \mathbb{C}^{I_1\times\cdots\times I_M \times J_1 \times\cdots\times J_N} $ is denoted by $ \mc{A}^{\#}_{\mc{N},\mc{M}}$ and defined as   $ \mc{A}^{\#}_{\mc{N},\mc{M}} =  \mc{N}^{-1} *_N\mc{A}^* *_M\mc{M} $.
\end{definition}

  Next we present the properties of the weighted conjugate transpose of any arbitrary-order tensor, $\mc{A} \in \mathbb{C}^{I_1\times\cdots\times I_M \times J_1 \times\cdots\times J_N}$, as follows.  
    
\begin{lemma}\label{34}
Let $\mc{A} \in \mathbb{C}^{I_1\times\cdots\times I_M \times J_1 \times\cdots\times J_N}$, $\mc{B} \in \mathbb{C}^{J_1\times\cdots\times J_N \times K_1 \times\cdots\times K_L}$ 
and  Hermitian positive definite tensors $\mc{M} \in \mathbb{C}^{I_1\times\cdots\times I_M \times I_1 \times\cdots\times I_M}$, $\mc{P} \in \mathbb{C}^{K_1\times\cdots\times K_L \times K_1 \times\cdots\times K_L}$  and $\mc{N} \in \mathbb{C}^{J_1\times\cdots\times J_N \times J_1 \times\cdots\times J_N}$.Then\\
(a) $ (\mc{A}^\#_{\mc{N},\mc{M}})^\#_{\mc{M},\mc{N}} = \mc{A} $,\\
(b) $ (\mc{A}\n\mc{B})^\#_{\mc{P},\mc{M}} = \mc{B}^\#_{\mc{P},\mc{N}}\n\mc{A}^\#_{\mc{N},\mc{M}} $.
\end{lemma}     


Adopting the result of Lemma~\ref{34}(b) and the definition of the weighted Moore-Penrose inverse, we can write the following identities. 

 \begin{lemma}
 Let $\mc{A}\in \mathbb{C}^{I_1\times\cdots\times I_M \times J_1
\times\cdots\times J_N }$, and $ \mc{M} \in \mathbb{C}^{I_1\times\cdots\times I_M \times I_1 \times\cdots\times I_M}$ , $\mc{N} \in \mathbb{C}^{J_1\times\cdots\times J_N \times J_1 \times\cdots\times J_N} $  are Hermitian positive definite tensors. 
Then
\begin{enumerate}
\item[(a)] $ (\mc{A}^\#_{\mc{N},\mc{M}})^\dg_{\mc{N},\mc{M}} 
= (\mc{A}^\dg_{\mc{M},\mc{N}})^\#_{\mc{M},\mc{N}}   $
\item[(b)] $\mc{A} = \mc{A} \n \mc{A}^\#_{\mc{N},\mc{M}} \m (\mc{A}^\#_{\mc{N},\mc{M}})^{\dg}_{\mc{N},\mc{M}} 
= (\mc{A}^\#_{\mc{N},\mc{M}})^{\dg}_{\mc{N},\mc{M}} \n \mc{A}^\#_{\mc{N},\mc{M}} \m \mc{A};$
\item[(c)] $\mc{A}^\#_{\mc{N},\mc{M}} = \mc{A}^{\dg}_{\mc{M},\mc{N}} \m \mc{A} \n \mc{A}^\#_{\mc{N},\mc{M}} 
= \mc{A}^\#_{\mc{N},\mc{M}} \m \mc{A} \n \mc{A}^{\dg}_{\mc{M},\mc{N}}$.
\end{enumerate}
 \end{lemma}
 
Using the Lemma 3.17 in \cite{PanRad18} on two invertible tensors $\mc{B}  \in \mathbb{C}^{I_1\times\cdots\times I_M \times I_1 \times\cdots\times I_M} $ and $\mc{C}  \in \mathbb{C}^{J_1\times\cdots\times J_N \times J_1 \times\cdots\times J_N}$, one can write the following identities
\begin{eqnarray}\label{KRDresult}
 (\mc{B}\m \mc{A})^\dg \m \mc{B}\m \mc{A} = \mc{A}^\dg \m \mc{A} ~~and~~
\mc{A}*_N \mc{C}\n (\mc{A}\n \mc{C})^\dg =\mc{A} \n \mc{A}^\dg
\end{eqnarray}
where $\mc{A}$ is the arbitrary-order tensor, i.e., $\mc{A} \in \mathbb{C}^{I_1\times\cdots\times I_M \times J_1 \times\cdots\times J_N}$. By  Eq.\eqref{KRDresult} and Corollary~\ref{3.4}, we get following results.

\begin{lemma}\label{42}
Let $\mc{A} \in \mathbb{C}^{I_1\times\cdots\times I_M \times J_1 \times\cdots\times J_N} $, and $\mc{M}  \in \mathbb{C}^{I_1\times\cdots\times I_M \times I_1 \times\cdots\times I_M} $, $\mc{N} \in \mathbb{C}^{J_1\times\cdots\times J_N \times J_1 \times\cdots\times J_N}$ be a pair of Hermitian positive definite tensors. Then\\
(a)~$\mc{A}_{\mc{M},\mc{I}_N}^\dg *_M \mc{A} =(\mc{M}^{1/2} *_M\mc{A})^\dg *_M \mc{M}^{1/2} *_M \mc{A} = \mc{A}^\dg *_M \mc{A} $,\\
(b)~$ \mc{A} *_N \mc{A}_{\mc{I}_M,\mc{N}}^\dg =\mc{A} *_N \mc{N}^{-1/2} *_N(\mc{A} *_N \mc{N}^{-1/2})^\dg = \mc{A} *_N \mc{A}^\dg $,\\
(c)~$ (\mc{A}_{\mc{M},\mc{I}_N}^\dag)^* = \mc{M}^{1/2} *_M [\mc{A}^* *_M \mc{M}^{1/2}]^\dg $,\\
(d)~$ (\mc{A}_{\mc{I}_M,\mc{N}}^\dag)^* =(\mc{N}^{-1/2}*_N\mc{A}^*)^\dg*_N \mc{N}^{-1/2} $.
\end{lemma}

The considerable amount of conventional and important facts  with the properties concerning the range space of arbitrary-order tensor, the following theorem obtains the well-formed result.

\begin{theorem}\label{ABThe45}
Let $\mc{U}\in \mathbb{C}^{I_1\times\cdots\times I_M \times J_1
\times\cdots\times J_N }$, $\mc{V}\in \mathbb{C}^{J_1\times\cdots\times J_N \times K_1\times\cdots\times K_L }$. Let $\mc{M}  \in \mathbb{C}^{I_1\times\cdots\times I_M \times I_1 \times\cdots\times I_M}$ and $\mc{N} \in \mathbb{C}^{K_1\times\cdots\times k_L \times K_1 \times\cdots\times K_L}$ be a pair of Hermitian positive definite tensors. Then
\begin{equation*}
(\mc{U} *_N \mc{V})_{\mc{M},\mc{N}}^\dg =[(\mc{U}_{\mc{M},\mc{I}_N}^\dg)^* *_N \mc{V}]_{\mc{M}^{-1},\mc{N}}^\dg *_M (\mc{V}_{\mc{I}_N,\mc{N}}^\dg *_N \mc{U}_{\mc{M},\mc{I}_N}^\dg)^* *_L [\mc{U} *_N(\mc{V}_{\mc{I}_N,\mc{N}}^\dg) ^ *]_{\mc{M},\mc{N}^{-1}}^\dg.
\end{equation*}
\end{theorem}

\begin{proof}
 Let $\mc{X} =(\mc{U}^\dag)^* *_N\mc{V}$ and $\mc{Y}=\mc{U}*_N(\mc{V}^\dag)^*$.
From Lemma~\ref{IDR1}(c) we get 
\begin{equation*}
\mathfrak{R}(\mc{X}^*) 
=  \mathfrak{R}[(\mc{U}*_N\mc{V})^*]\textnormal{~~and~~}  \mathfrak{R}(\mc{Y})
 = \mathfrak{R}(\mc{U}*_N\mc{V}).
\end{equation*}
Now, using Lemma~\ref{IDR1}(b) and Lemma~\ref{ID}(b) along with the fact $(\mc{V}^\dag)^*= \mc{V}*_L\mc{V}^\dag *_N(\mc{V}^\dag)^*$,  we obtain,  
\begin{eqnarray*}
\mc{X}^\dag *_M(\mc{V}^\dag *_N\mc{U}^\dag)^* *_L\mc{Y}^\dag 
&=& \mc{X}^\dag *_M(\mc{U}^\dag)^* \n\mc{V}*_L\mc{V}^\dag *_N(\mc{V}^\dag)^* *_L\mc{Y}^\dg \\
&=& \mc{X}^\dag *_M\mc{X}*_L\mc{V}^\dag *_N(\mc{V}^\dag)^* *_L\mc{Y}^\dag\\
    &=& (\mc{U}*_N\mc{V})^\dag *_M\mc{Y}*_L\mc{Y}^\dag 
    =(\mc{U}*_N\mc{V})^\dag.
    \end{eqnarray*}
    Replacing $ \mc{U}$ and $\mc{V} $ by $ \mc{M}^{1/2}\m\mc{U} $  and $\mc{V}*_L \mc{N}^{-1/2}  $  respectively on the above result, we get
\begin{eqnarray*}
[(\mc{M}^{1/2} *_M &&\hspace{-.3cm}\mc{U})*_N(\mc{V}*_L\mc{N}^{-1/2})]^{\dg}\\
&=&\{[(\mc{M}^{1/2}*_M\mc{U})^{\dg}]^{*}*_N\mc{V}*_L\mc{N}^{-1/2}\}^{\dg}*_M[(\mc{M}^{1/2}*_M\mc{U})^{\dg}]^{*}*_N[(\mc{V}*_L\mc{N}^{-1/2})^{\dg}]^{*}*_L\\
&&\hspace{2cm}\{\mc{M}^{1/2}*_M\mc{U}\n[(\mc{V}*_L\mc{N}^{-1/2})^{\dg}]^{*}\}^{\dg}\\
&=&[\mc{M}^{-1/2}*_M(\mc{U}_{\mc{M},\mc{I}_N}^{\dg})^{*}*_N\mc{V}*_L\mc{N}^{-1/2}]^{\dg}*_M\mc{M}^{-1/2}*_M(\mc{U}_{\mc{M},\mc{I}_N}^{\dg})^{*}*_N
(\mc{V}_{\mc{I}_N,\mc{N}}^{\dg})^{*}*_L
\\
&&\hspace{2cm}
\mc{N}^{1/2}*_L[\mc{M}^{1/2}*_M\mc{U}*_N(\mc{V}_{\mc{I}_N,\mc{N}}^{\dg})^{*}*_L\mc{N}^{1/2}]^{\dg}.
\end{eqnarray*}
Substituting the above result in Eq.\eqref{2.17} we get the desired result. 
\end{proof}

Further, in connection with range space of arbitrary-order tensor, the following Theorem collects some useful identities of weighted Moore-Penrose inverses.

\begin{theorem}\label{A1}
Let $ \mc{U}\in \mathbb{C}^{I_1\times\cdots\times I_M \times J_1
\times\cdots\times J_N },~~\mc{V}\in \mathbb{C}^{J_1\times\cdots\times J_N \times K_1
\times\cdots\times K_L }$ and~ $\mc{W}\in \mathbb{C}^{K_1\times\cdots\times K_L \times H_1
\times\cdots\times H_R }$. If $ \mc{A} = \mc{U}\n\mc{V}\1\mc{W}$, where  $ \mc{M}  \in \mathbb{C}^{I_1\times\cdots\times I_M \times I_1 \times\cdots\times I_M} $ and $ \mc{N} \in \mathbb{C}^{H_1\times\cdots\times H_R \times H_1 \times\cdots\times H_R} $ are Hermitian positive definite tensors. Then 
\begin{enumerate}
    \item [(a)] $\mc{A}^\dag_{\mc{M},\mc{N}} = \mc{X}^\dg_{\mc{I}_N,\mc{N}}\n\mc{V}\1 \mc{Y}^\dg_{\mc{M},\mc{I}_L} $, where $\mc{X} = (\mc{U}\n\mc{V}\1\mc{V}^\dg)^\dg\m\mc{A} $ and $ \mc{Y} = \mc{A}\f(\mc{V}^\dg\n\mc{V}\1\mc{W})^\dg $;
    \item[(b)] $\mc{A}^\dag_{\mc{M},\mc{N}} = \mc{X}^\dg_{\mc{I}_L,\mc{N}}\1\mc{V}^*\n\mc{V} \1 \mc{V}^*\n \mc{Y}^\dg_{\mc{M},\mc{I}_N} $, where $ \mc{X} = [\mc{U}\n(\mc{V}^\dg)^*]^\dg\m\mc{A} $ and $ \mc{Y} = \mc{A}\f[(\mc{V}^\dg)^*\1\mc{W}]^\dg $.
    \end{enumerate}
\end{theorem}

\begin{proof}
(a) From Eq.\eqref{2.17} we have
\begin{equation*}
 \mc{A}^\dg _{\mc{M},\mc{N}}  
 =  \mc{N}^{-1/2}\f[\mc{U}_1\n\mc{V}\1\mc{W}_1]^\dg \m \mc{M}^{1/2}, 
\end{equation*}
where $\mc{U}_1 = \mc{M}^{1/2}\m\mc{U} $ and $ \mc{W}_1 = \mc{W}\f\mc{N}^{-1/2} $.
On the other hand, by Eq. \eqref{lemma37}, we have
\begin{equation*}
\mathfrak{R}(\mc{X}^*) = \mathfrak{R}[(\mc{V}\1\mc{W})^*\n(\mc{U}\n\mc{V}\1\mc{V}^\dg)^*\m\{(\mc{U}\n\mc{V}\1\mc{V}^\dg)^*\}^\dg]  = \mathfrak{R}(\mc{A}^*) \textnormal{~~and~~} \mathfrak{R}(\mc{Y}) = \mathfrak{R}(\mc{A}).
\end{equation*}
Also by Lemma \ref{ID}(c), we get
\begin{eqnarray*}
 \mathfrak{R}[(\mc{X}^\dg)^*] 
 \subseteq \mathfrak{R}[(\mc{U}\n\mc{V}\1\mc{V}^\dg)^\dg] = \mathfrak{R}[(\mc{U}\n\mc{V}\1\mc{V}^\dg)^*]~~\textnormal{and~~}
\mathfrak{R}(\mc{Y}^\dg) 
\subseteq \mathfrak{R}(\mc{V}^\dg\n\mc{V}\1\mc{W}).
\end{eqnarray*}
Thus, by using Lemma~\ref{ID}[(a),(b)] and Lemma~\ref{IDR1}[(a),(b)], we have
\begin{eqnarray}\label{3.10}
\mc{X}^\dg \n \mc{V}\1 \mc{Y}^\dg \nonumber
&=& \mc{X}^\dg\n(\mc{U}\n\mc{V}\1\mc{V}^\dg)^\dg\m(\mc{U}\n\mc{V}\1\mc{V}^\dg)\n \mc{V}\1 (\mc{V}^\dg\n\mc{V}\1\mc{W})\f \\ \nonumber
&&\hspace{3cm}(\mc{V}^\dg\n\mc{V}\1\mc{W})^\dg \n \mc{Y}^\dg\\
&=& \mc{A}^\dg\m\mc{Y}\n\mc{Y}^\dg
= \mc{A}^\dg. 
\end{eqnarray}
Replacing $\mc{U}$ by $\mc{U}_1$ and $\mc{W}$ by $ \mc{W}_1$ in Eq.\eqref{3.10} and then using Lemma~\ref{42}[(a),(b)], we get 
\begin{eqnarray*}
\mc{A}^\dg _{\mc{M},\mc{N}}
&=& \mc{N}^{-1/2}\f[(\mc{U}_1\n\mc{V}\1\mc{V}^\dg)^\dg\m\mc{M}^{1/2}\m \mc{A}\f\mc{N}^{-1/2}]^\dg \n\mc{V}\1\\
&&\hspace{4cm}[\mc{M}^{1/2}\m \mc{A}\f \mc{N}^{-1/2}\f(\mc{V}^\dg\n\mc{V}\1\mc{W}_1)^\dg]^\dg \m \mc{M}^{1/2}\\
&=&\mc{N}^{-1/2}\f[(\mc{U}\n\mc{V}\1\mc{V}^\dg)^\dg\m \mc{A}\f\mc{N}^{-1/2}]^\dg \n\mc{V}\1[\mc{M}^{1/2}\m \mc{A}\f(\mc{V}^\dg\n\mc{V}\1\mc{W})^\dg]^\dg \m \mc{M}^{1/2} \\
&=&  \mc{X}^\dg_{\mc{I}_N,\mc{N}}\n\mc{V}\1 \mc{Y}^\dg_{\mc{M},\mc{I}_L}.
\end{eqnarray*}

(b) Following the Lemma~\ref{revA}(b) and Eq. \eqref{lemma37}, we get \\
$\mathfrak{R}(\mc{A})= \mathfrak{R}(\mc{Y})$ and 
$\mathfrak{R}(\mc{X}^*) = \mathfrak{R}[(\mc{V}^* \n \mc{V}*_L\mc{W})^* *_L(\mc{U}*_N(\mc{V}^\dag)^*)^* *_M\{( \mc{U}*_N(\mc{V}^\dag)^*)^*\}^\dag] = \mathfrak{R}(\mc{A}^*)$.
Also by using Lemma \ref{ID}(c)
\begin{eqnarray*}
\mathfrak{R}[(\mc{X}^\dag)^*] 
\subseteq \mathfrak{R}[\{ \mc{U}*_N(\mc{V}^\dag)^*\}^*] \textnormal{~~~and~~~}
 \mathfrak{R}(\mc{Y}^\dag)
\subseteq \mathfrak{R}[(\mc{V}^\dag)^* *_L\mc{W}].
\end{eqnarray*}
Using Lemma \ref{ID}[(a),(b)] and Lemma \ref{IDR1}[(a),(b)], we obtain
\begin{eqnarray}\nonumber
 \mc{X}^\dg\1\mc{V}^*\n\mc{V} \1 \mc{V}^*\n \mc{Y}^\dg
&=& \mc{X}^\dag *_L[\mc{U}*_N(\mc{V}^\dag)^*]^\dag *_M[\mc{U}*_N(\mc{V}^\dag)^*]\1\mc{V}^*\n\mc{V} \1 \mc{V}^*\n\\\nonumber
&&\hspace{4cm}[(\mc{V}^\dag)^* *_L\mc{W}]*_R[(\mc{V}^\dag)^* *_L\mc{W}]^\dag *_N\mc{Y}^\dag \\\nonumber
&=& \mc{X}^\dag *_L[\mc{U}*_N(\mc{V}^\dag)^*]^\dag *_M \mc{A}\f [(\mc{V}^\dag)^* *_L\mc{W}]^\dag *_N \mc{Y}^\dag \\ \label{eq282}
&=& \mc{A}^\dg \m\mc{A}\f[(\mc{V}^\dag)^* *_L\mc{W}]^\dag *_N\mc{Y}^\dg
=\mc{A}^\dg. 
\end{eqnarray}
Let $ \mc{U}_1 = \mc{M}^{1/2}\m \mc{U},~~ \mc{W}_1 = \mc{W}\f \mc{N}^{-1/2} $ and $ \mc{A}_1 = \mc{U}_1\n \mc{V}\1\mc{W}_1 $. 
Then using Eq. \eqref{eq282} and the Lemma~\ref{42} [(a),(b)], we can write
\begin{eqnarray*}
\mc{A}_1 ^\dg 
&=& [\{\mc{U}_1\n(\mc{V}^\dg)^*\}^\dg\m\mc{A}_1]^\dg \1 \mc{V}^*\n\mc{V} \1 \mc{V}^*\n[\mc{A}_1\f\{(\mc{V}^\dg)^*\1\mc{W}_1\}^\dg]^\dg\\
&=& [\{\mc{U}\n(\mc{V}^\dg)^*\}^\dg\m\mc{A}\f \mc{N}^{-1/2} ]^\dg \1 \mc{V}^*\n\mc{V} \1 \mc{V}^*\n[\mc{M}^{1/2}\m \mc{A}\f\{(\mc{V}^\dg)^*\1\mc{W}\}^\dg]^\dg.
\end{eqnarray*}
Therefore, $ \mc{A}^\dg_{\mc{M},\mc{N}} = \mc{N}^{-1/2} \f \mc{A}_1 ^\dg \m \mc{M}^{1/2} = \mc{X}^\dg_{\mc{I}_L,\mc{N}}\1\mc{V}^*\n\mc{V} \1 \mc{V}^*\n \mc{Y}^\dg_{\mc{M},\mc{I}_N}  $.
\end{proof}

\begin{theorem}
Let $\mc{U}\in \mathbb{C}^{I_1\times\cdots\times I_M \times J_1
\times\cdots\times J_N }$, $ \mc{V}\in \mathbb{C}^{J_1\times\cdots\times J_N \times K_1
\times\cdots\times K_L }$ and $ \mc{W}\in \mathbb{C}^{K_1\times\cdots\times K_L \times H_1
\times\cdots\times H_R} $.  Also let $ \mc{M}  \in \mathbb{C}^{I_1\times\cdots\times I_M \times I_1 \times\cdots\times I_M} $ and $ \mc{N} \in \mathbb{C}^{H_1\times\cdots\times H_R \times H_1 \times\cdots\times H_R} $ be a pair of Hermitian positive definite tensors. Then 
\begin{enumerate}
     \item [(a)] $  (\mc{U}*_N\mc{V}*_L\mc{W})_{\mc{M},\mc{N}}^{\dg} = [(\mc{U}_{\mc{M},\mc{I}_N}^{\dg})^{*}*_N\mc{V}*_L\mc{W}]_{\mc{M}^{-1},\mc{N}}^{\dg}*_M(\mc{U}_{\mc{M},\mc{I}_N}^{\dg})^{*}*_N\mc{V}*_L(\mc{W}_{\mc{I}_L,\mc{N}}^{\dg})^{*}\f[\mc{U}*_N\mc{V}*_L(\mc{W}_{\mc{I}_L,\mc{N}}^{\dg})^{*}]_{\mc{M},\mc{N}^{-1}}^{\dg}  $;
    \item[(b)]  $ (\mc{U}*_N\mc{V}*_L\mc{W})_{\mc{M},\mc{N}}^{\dg} 
 = [\{(\mc{U}*_N\mc{V})_{\mc{M},\mc{I}_L}^{\dg}\}^{*}*_L\mc{W}]_{\mc{M}^{-1},\mc{N}}^{\dg}*_M[(\mc{U}*_N\mc{V})_{\mc{M},\mc{I}_L}^{\dg}]^{*}*_L\mc{V}^{\dg}*_N[(\mc{V}*_L\mc{W})_{\mc{I}_N,\mc{N}}^{\dg}]^{*}*_R[\mc{U}*_N\{(\mc{V}*_L\mc{W})_{\mc{I}_N,\mc{N}}^{\dg}\}^{*}]_{\mc{M},\mc{N}^{-1}}^{\dg} $.
\end{enumerate}

\end{theorem}

\begin{proof}
(a) Let $ \mc{A} = \mc{U}\n\mc{V}\1\mc{W}$,
$ \mc{X} = (\mc{U}^\dag)^*\n \mc{V}\1\mc{W} $ and $ \mc{Y}= \mc{U}\n\mc{V}\1(\mc{W}^\dag)^* $.\\
Using Lemma \ref{IDR1}(c), we get
\begin{eqnarray*}
\mathfrak{R}(\mc{X}^*) 
= \mathfrak{R}(\mc{A}^*) \textnormal{~~~and~~~} 
 \mathfrak{R}(\mc{Y}) 
 = \mathfrak{R}(\mc{A}).
\end{eqnarray*}
Further, using the fact that $\mc{W}^\dag =\mc{W}^\dag*_N(\mc{W}^\dag)^* *_N\mc{W}^*$,and Lemma \ref{ID}(b) and Lemma \ref{IDR1}(b), we can write,
\begin{eqnarray*}
\mc{X}^\dag*_M(\mc{U}^\dag)^* *_N\mc{V}*_L(\mc{W}^\dag)^* *_R\mc{Y}^\dag 
 &=& \mc{X}^\dag *_M(\mc{U}^\dag)^* *_N\mc{V}*_L\mc{W}*_R\mc{W}^\dag *_L(\mc{W}^\dag)^* *_R\mc{Y}^\dag\\
 &=& 
 \mc{A}^\dag \m \mc{Y} \f \mc{Y}^\dag
 =\mc{A}^\dag.
\end{eqnarray*}

Using the above result to  
$[({\mc{M}^{1/2} *_M\mc{U})\n\mc{V}\1(\mc{W}\f\mc{N}^{-1/2})}]^\dg $ and following the Lemma~\ref{42}[(c),(d)] we get,
\begin{eqnarray*}
({\mc{U}*_N\mc{V}*_L\mc{W}})_{\mc{M},\mc{N}}^\dg 
&=& \mc{N}^{-1/2}\f \{[(\mc{M}^{1/2}*_M\mc{U})^{\dg}]^{*}*_N\mc{V}*_L\mc{W}*_R\mc{N}^{-1/2}\}^{\dg}*_M[(\mc{M}^{1/2}*_M\mc{U})^{\dg}]^{*}*_N \mc{V} *_L \hspace{2cm} \\
&& [(\mc{W}*_R\mc{N}^{-1/2})^{\dg}]^{*} *_R[\mc{M}^{1/2} *_M\mc{U} *_N\mc{V}*_L
\{(\mc{W}*_R\mc{N}^{-1/2})^\dg\}^*]^\dg \m \mc{M}^{1/2} \\
&=& \mc{N}^{-1/2}\f [\mc{M}^{-1/2}*_M(\mc{U}_{\mc{M},\mc{I}_N}^{\dg})^{*}*_N\mc{V}*_L\mc{W}*_R\mc{N}^{-1/2}]^{\dg}*_M{\mc{M}}^{-1/2}*_M(\mc{U}_{\mc{M},\mc{I}_N}^\dag)^**_N \\
&& \mc{V}*_L(\mc{W}_{\mc{I}_L,\mc{N}}^{\dg})^{*}*_R{\mc{N}^{1/2}}*_R[\mc{M}^{1/2}*_M\mc{U}*_N\mc{V}*_L(\mc{W}_{\mc{I}_L,\mc{N}}^{\dg})^{*}*_R\mc{N}^{1/2}]^{\dg} \m \mc{M}^{1/2}\\
&=& [(\mc{U}_{\mc{M},\mc{I}_N}^{\dg})^{*}*_N\mc{V}*_L\mc{W}]_{\mc{M}^{-1},\mc{N}}^{\dg}*_M(\mc{U}_{\mc{M},\mc{I}_N}^{\dg})^{*}*_N\mc{V}*_L
(\mc{W}_{\mc{I}_L,\mc{N}}^{\dg})^{*}\f \\
&& [\mc{U}*_N\mc{V}*_L(\mc{W}_{\mc{I}_L,\mc{N}}^{\dg})^{*}]_{\mc{M},\mc{N}^{-1}}^{\dg}.
\end{eqnarray*}


(b) Let $ \mc{A} = \mc{U}\n\mc{V}\1\mc{W}$, $ \mc{X} = [(\mc{U}*_N\mc{V})^\dag]^* *_L\mc{W} $  and $ \mc{Y} = \mc{U}*_N[(\mc{V}*_L\mc{W})^\dag]^* $. 
Using Lemma\ref{IDR1}(c), we get  \begin{eqnarray*}
&&\mathfrak{R}(\mc{X}^*) 
= \mathfrak{R}(\mc{A}^*),~~~
\mathfrak{R}(\mc{Y}) = \mathfrak{R}(\mc{A})
\textnormal{~~and~~}  \\ 
&&\mathfrak{R}[\{(\mc{U}\n\mc{V})^\dg\}^*]^*  = \mathfrak{R}[(\mc{U}\n\mc{V})^\dg] = \mathfrak{R}[(\mc{U}\n \mc{V})^*] \subseteq \mathfrak{R}(\mc{V}^*).
\end{eqnarray*}
From Lemma~\ref{revA}(a), one can write $ [(\mc{V}\1\mc{W})^\dg]^* = (\mc{V}\1\mc{W})\f(\mc{V}\1\mc{W})^\dg\n[(\mc{V}\1\mc{W})^\dg]^* $.\\
Now using Lemma \ref{ID}(b) and Lemma \ref{IDR1}(b), we obtain
\begin{eqnarray*}
 \mc{X}^\dag \m & & \hspace{-.3cm} [(\mc{U}\n \mc{V})^\dag]^* *_L\mc{V}^\dag *_N[(\mc{V}*_L\mc{W})^\dag]^* *_R\mc{Y}^\dag\\
   & = &\mc{X}^\dag *_M[(\mc{U}*_N\mc{V})^\dag]^* *_L\mc{V}^\dag *_N\mc{V}*_L\mc{W}*_R(\mc{V}*_L\mc{W})^\dag *_N[(\mc{V}*_L\mc{W})^\dag]^* \f \mc{Y}^\dag\\
& = & \mc{X}^\dag *_M\mc{X}*_R(\mc{V}*_L\mc{W})^\dag *_N[(\mc{V}*_L\mc{W})^\dag]^* *_R\mc{Y}^\dag\\    
& = & \mc{A}^\dag *_M\mc{U}*_N[(\mc{V}*_L\mc{W})^\dag]^* *_R\mc{Y}^\dag
 = \mc{A}^\dag.
\end{eqnarray*}
Replacing $ \mc{U}$ and $\mc{W} $ by $ \mc{M}^{1/2}\m\mc{U} $ and $\mc{W}\f \mc{N}^{-1/2} $ respectively on the above result, we have
\begin{eqnarray*}
& & \hspace{-1cm} [{(\mc{M}^{1/2}*_M\mc{U})*_N\mc{V}*_L (\mc{W}*_R\mc{N}^{-1/2})}]^\dag\\
&=&[\{(\mc{M}^{1/2}*_M\mc{U}*_N\mc{V})^\dag\}^* *_L\mc{W}*_R\mc{N}^{-1/2}]^\dag *_M[(\mc{M}^{1/2}*_M\mc{U}*_N\mc{V})^\dag]^* *_L\mc{V}^\dag\\  
& & \hspace{3cm} *_N[(\mc{V}*_L\mc{W}*_R\mc{N}^{-1/2})^\dag]^*
*_R [\mc{M}^{1/2}*_M\mc{U}*_N\{(\mc{V}*_L\mc{W}*_R\mc{N}^{-1/2})^\dag\}^*]^\dag\\
&=&\{\mc{M}^{-1/2}\m[(\mc{U}*_N\mc{V})_{\mc{M},\mc{I}_L}^\dag]^* *_L\mc{W}*_R\mc{N}^{-1/2}\}^\dag *_M\mc{M}^{-1/2}*_M[(\mc{U}*_N\mc{V})_{\mc{M},\mc{I}_L}^\dag]^* *_L\mc{V}^\dag \\
& & \hspace{3cm} *_N[(\mc{V}*_L\mc{W})_{\mc{I}_N, \mc{N}}^\dag]^* *_R\mc{N}^{-1/2} *_R[\mc{M}^{1/2}*_M\mc{U}*_N\{(\mc{V}*_L\mc{W})_{\mc{I}_N,\mc{N}}^\dag\}^* *_R\mc{N}^{1/2}]^\dag .
\end{eqnarray*}
Substituting the above result in Eq.\eqref{2.17} one can get the desired result. 
\end{proof}

\begin{theorem}\label{uvw1}
Let $ \mc{U}\in \mathbb{C}^{I_1\times\cdots\times I_M \times J_1
\times\cdots\times J_N} $, $ \mc{V}\in \mathbb{C}^{J_1\times\cdots\times J_N \times K_1
\times\cdots\times K_L} $ and $ \mc{W} \in  \mathbb{C}^{K_1\times\cdots\times K_L \times H_1
\times\cdots\times H_R} $. If $ \mc{A} = \mc{U}\n\mc{V}\1\mc{W} $, and $\mc{M}  \in \mathbb{C}^{I_1\times\cdots\times I_M \times I_1 \times\cdots\times I_M} $, $\mc{N} \in \mathbb{C}^{H_1\times\cdots\times H_R \times H_1 \times\cdots\times H_R} $ be a pair of Hermitian positive definite tensors, Then 
\begin{equation*}
\mc{A}^\dag_{\mc{M},\mc{N}} = \mc{X}^\dg_{\mc{I}_N,\mc{N}}\n\mc{V}\1 \mc{Y}^\dg_{\mc{M},\mc{I}_L}, 
\end{equation*}
 where $\mc{X} = \mc{U}^\dg\m\mc{A} $ and $ \mc{Y} = \mc{A}\f\mc{W}^\dg $.
\end{theorem}

\begin{proof}
Let $ \mc{U}_1 = \mc{M}^{1/2}\m \mc{U} $ and $ \mc{W}_1 = \mc{W}\f \mc{N}^{-1/2} $. It is known, from Eq.\eqref{2.17},
\begin{equation*}
\mc{A}^\dg_{\mc{M},\mc{N}} = \mc{N}^{-1/2}\f(\mc{U}_1 \n\mc{V}\1\mc{W}_1)^\dg \f \mc{N}^{-1/2}.
\end{equation*}
Now using Eq.\eqref{lemma37} we have
$\mathfrak{R}(\mc{X}^*) = \mathfrak{R}[(\mc{U}*_N\mc{V}*_L\mc{W})^* *_N(\mc{U}^*)^\dag] =\mathfrak{R}(\mc{A}^*)
$  and
 $ \mathfrak{R}(\mc{Y}) = \mathfrak{R}(\mc{A})$.
 Also, from Lemma \ref{ID} (c), we have   
\begin{eqnarray*}
 \mathfrak{R}[(\mc{X}^\dag)^*] = \mathfrak{R}(\mc{X}) \subseteq \mathfrak{R}(\mc{U}^\dag) = \mathfrak{R}(\mc{U}^*) \textnormal{~and~} \mathfrak{R}(\mc{Y}^\dag) =\mathfrak{R}(\mc{Y}^*) = \mathfrak{R}(\mc{W}*_R\mc{W}^\dag *_L(\mc{U}*_N\mc{V})^*) \subseteq \mathfrak{R}(\mc{W}).
\end{eqnarray*}
Using Lemma~\ref{ID}[(a),(b)] and Lemma~\ref{IDR1}[(a),(b)], we get
 \begin{eqnarray}\nonumber
  \mc{X}^\dag *_N\mc{V} *_L\mc{Y}^\dag 
 &=& \mc{X}^\dag *_N\mc{U}^\dag*_M \mc{U}*_N\mc{V}*_L\mc{W}*_R\mc{W}^\dag *_L\mc{Y}^\dag \\\label{thm3.13}
 &=& \mc{A}^\dag *_N\mc{Y}*_L\mc{Y}^\dag
 = \mc{A}^\dag.
 \end{eqnarray}
 Using Lemma\ref{42}[(a),(b)]) one can conclude 
 \begin{eqnarray*}
 (\mc{U}_1 \n \mc{V} \1 \mc{W}_1)^\dg &=&  [\mc{U}_1^\dg \m \mc{M}^{1/2}\m \mc{A}\f \mc{N}^{-1/2} ]^\dg \n \mc{V}\1 [\mc{M}^{1/2}\m \mc{A}\f \mc{N}^{-1/2}\f \mc{W}_1 ^\dg]^\dg\\
 &=& [\mc{U}^\dg \m \mc{A}\f \mc{N}^{-1/2} ]^\dg \n \mc{V}\1 [\mc{M}^{1/2}\m \mc{A}\f \mc{W}^\dg]^\dg.
 \end{eqnarray*}
 \begin{eqnarray*}
\textnormal{Hence,}  ~  \mc{A}^\dg_{\mc{M},\mc{N}}
&=&   \mc{N}^{-1/2}\f [\mc{U}^\dg \m \mc{A}\f \mc{N}^{-1/2} ]^\dg \n \mc{V}\1 [\mc{M}^{1/2}\m \mc{A}\f \mc{W}^\dg]^\dg \m \mc{M}^{1/2}\\
&=& \mc{X}^\dg_{\mc{I}_N,\mc{N}}\n \mc{V}\1\mc{Y}^\dg_{\mc{M},\mc{I}_L}.
\end{eqnarray*}
Hence the proof is complete.
\end{proof}

By Lemma~\ref{42}[(a),(b)] and Eq. \eqref{2.17}, and Eq. \eqref{thm3.13} we have 
\begin{corollary}
Let $ \mc{U}\in \mathbb{C}^{I_1\times\cdots\times I_M \times J_1
\times\cdots\times J_N }$, $\mc{V}\in \mathbb{C}^{J_1\times\cdots\times J_N \times K_1 \times\cdots\times K_L }$ and $\mc{W}\in \mathbb{C}^{K_1\times\cdots\times K_L \times H_1\times\cdots\times H_R }$. Let $ \mc{A} = \mc{U}\n \mc{V} \1 \mc{W} $. Let $ \mc{M} \in \mathbb{C}^{I_1\times\cdots\times I_M \times I_1 \times\cdots\times I_M}$, $\mc{N} \in \mathbb{C}^{H_1\times\cdots\times H_R \times H_1 \times\cdots\times H_R} $, 
$ \mc{P} \in \mathbb{C}^{J_1\times\cdots\times J_N \times J_1 \times\cdots\times J_N} $ and $ \mc{Q} \in \mathbb{C}^{K_1\times\cdots\times K_L \times K_1 \times\cdots\times K_L} $
 are Hermitian positive definite tensors. Then the weighted Moore-Penrose inverse of  $ \mc{A} $ with respect to $\mc{M}$ { and} $\mc{N}$ { satisfies the following identities}:  \\
(a) $\mc{A}_{\mc{M},\mc{N}}^{\dg}
=(\mc{U}_{\mc{I}_M,\mc{P}}^{\dg}\m\mc{A})_{\mc{P},\mc{N}}^{\dg}\n\mc{V}\1(\mc{A}\f \mc{W}_{\mc{Q},\mc{I}_R}^{\dg})_{\mc{M},\mc{Q}}^{\dg},$\\
(b) $ \mc{A}_{\mc{M},\mc{N}}^{\dg}
= [(\mc{U}*_N\mc{V}*_L\mc{V}_{\mc{P},\mc{I}_L}^{\dg})_{\mc{M},\mc{P}}^{\dg}*_M\mc{A}]_{\mc{P},\mc{N}}^{\dg}*_N\mc{V}*_L[\mc{A}\f(\mc{V}_{\mc{I}_N,\mc{Q}}^{\dg}*_N\mc{V}*_L\mc{W})_{\mc{Q},\mc{N}}^{\dg}]_{\mc{M},\mc{Q}}^{\dg}$.

\end{corollary}

\subsection{The full rank decomposition}
The tensors and their decompositions originally appeared in 1927, \cite{Hit27}. The idea of decomposing a tensor  as a product of tensors with a more desirable structure may well be one of the most important developments in numerical analysis such as the implementation of numerically efficient algorithms and the solution of multilinear systems \cite{kolda,Che2018,Marcar,Kolda01}.   As part of this section, we focus on the full rank decomposition of a tensor.  Unfortunately, It is very difficult to compute tensor rank. But the authors of in \cite{psmv18} introduced an useful and effective definition of the tensor rank, termed as reshaping rank. With the help of reshaping rank, We present one of our important results, full rank decomposition of an arbitrary-order tensor.

\begin{theorem}\label{49}
Let $ \mc{A} \in  \mathbb{C}^{I_1 \times \cdots \times I_M \times J_1 \times \cdots \times J_N} $. Then there exist a left invertible tensor 
$ \mc{F} \in \mathbb{C}^{I_1 \times \cdots \times I_M \times H_1 \times \cdots \times H_R} $ and a right invertible tensor $ \mc{G} \in \mathbb{C}^{H_1 \times \cdots \times H_R \times J_1 \times \cdots \times J_N} $  such that
\begin{equation} \label{37}
    \mc{A} = \mc{F} \f \mc{G},
\end{equation}
where  $ rshrank(\mc{F}) =rshrank(\mc{G})= rshrank(\mc{A}) = r = H_1 \cdots H_R $.  This is called the full rank decomposition of the tensor $ \mc{A} $.
\end{theorem}

\begin{proof}
 Let the matrix $A = rsh(\mc{A}) \in  \mathbb{C}^{I_1 \cdots I_M \times J_1 \cdots J_N}$. Then we have, rank(A) = r. 
 Suppose that the matrix $A$  has a full rank decompositions, as follows,
 \begin{equation}\label{frd}
A = FG,     
 \end{equation}
 where $ F \in \mathbb{C}^{I_1 \cdots I_M \times H_1 \cdots H_R} $ is a full column rank matrix and $ G \in \mathbb{C}^{H_1 \cdots H_R \times J_1 \cdots J_N} $ is a full row rank matrix.  From Eq. \eqref{77} and Eq. \eqref{frd} we obtain
 \begin{equation}\label{rem-FRF}
     rhs^{-1}(A)=rhs^{-1}(FG)= rhs^{-1}(F)*_R rhs^{-1}(G),
 \end{equation}
 where  $\mc{F} = rsh^{-1}(F) \in  \mathbb{C}^{I_1 \times \cdots \times I_M \times H_1 \times \cdots \times H_R} \textnormal{~ and~}  \mc{G} = rsh^{-1}(G) \in \mathbb{C}^{H_1 \times \cdots \times H_R \times J_1 \times \cdots \times J_N}.$
 it follows that
 \begin{equation*}
    \mc{A} = \mc{F} \f \mc{G},
\end{equation*}
where  $\mc{F} \in \mathbb{C}^{I_1 \times \cdots \times I_M \times H_1 \times \cdots \times H_R} $ is the left invertible tensor and $\mc{G} \in \mathbb{C}^{H_1 \times \cdots \times H_R \times J_1 \times \cdots \times J_N}$ is a right invertible tensor.
\end{proof}
The prove of the above theorem was proved earlier
(see  Lemma 2.3(a), \cite{LIANG2018}) indifferent way. Here, we have provided another proof without using reshape operation. Further, the authors of in \cite{LIANG2018} computed the Moore-Penrose inverse of a tensor using full rank decomposition of tensors.  as follows
\begin{lemma}{(Theorem 3.7, \cite{LIANG2018})} \label{MPComput}
If the full rank decomposition of a tensor $ \mc{A} \in  \mathbb{C}^{I_1 \times \cdots \times I_M \times J_1 \times \cdots \times J_N}$ is given as Theorem \ref{49}, then
\begin{equation}
\mc{A}^\dg = \mc{G}^**_R(\mc{F}^**_M\mc{A}*_N\mc{G}^*)^{-1}*_R\mc{F}^*.
\end{equation}
\end{lemma}
Now, the following theorem expressed the weighted Moore-Penrose inverse of a tensor 
$ \mc{A} \in  \mathbb{C}^{I_1 \times \cdots \times I_M \times J_1 \times \cdots \times J_N}$  in form of the ordinary tensor inverse. 
\begin{theorem}
If the full rank decomposition of a tensor $ \mc{A} \in \mathbb{C}^{I_1 \times \cdots \times I_M \times J_1 \times \cdots \times J_N} $ is given by Eq. \eqref{37}, then 
the weighted Moore-Penrose inverse of $\mc{A} $ can be written as
\begin{equation*}
    \mc{A}^\dg_{\mc{M},\mc{N}} = \mc{N}^{-1}\n \mc{G}^* \f (\mc{F}^* \m \mc{M} \m \mc{A} \n\mc{N}^{-1}\n \mc{G}^*)^{-1} \f \mc{F}^* \m \mc{M},
\end{equation*}
\end{theorem}
where  $ \mc{M}  \in \mathbb{C}^{I_1\times\cdots\times I_M \times I_1 \times\cdots\times I_M} $ and $ \mc{N} \in \mathbb{C}^{J_1\times\cdots\times J_N \times J_1 \times\cdots\times J_N} $ are Hermitian positive definite tensors.

\begin{proof}
 From Eq. \eqref{2.17}, we have  
 \begin{eqnarray*}
  \mc{A}^\dg_{\mc{M},\mc{N}} = \mc{N}^{-1/2}\n (\mc{M}^{1/2} \m\mc{A}\n\mc{N}^{-1/2})^\dg \m \mc{M}^{1/2} 
  = \mc{N}^{1/2}\n \mc{B}^\dg \m \mc{M}^{1/2},
 \end{eqnarray*}
where~$\mc{B} = (\mc{M}^{1/2} \m\mc{F}) \f (\mc{G} \n \mc{N}^{-1/2})$, ~and~  $ \mc{M} ~\&~ \mc{N}$ are Hermitian positive definite tensors. Here $\mc{B} $ is in the form of full rank decomposition, as both $\mc{M}^{1/2} $ and $ \mc{N}^{1/2} $ are invertible. Now, from Lemma~\ref{MPComput}, we get
\begin{eqnarray*}
 \mc{B}^\dg 
 &=& (\mc{G} \n \mc{N}^{-1/2})^* \f [(\mc{M}^{1/2} \m\mc{F})^* \m (\mc{M}^{1/2} \m\mc{F}) \f (\mc{G} \n \mc{N}^{-1/2}) \n (\mc{G} \n \mc{N}^{-1/2})^*]^{-1} \f (\mc{M}^{1/2} \m\mc{F})^* \\
 &=& \mc{N}^{-1/2}\n \mc{G}^* \f (\mc{F}^* \m\mc{M} \m \mc{A} \n \mc{N}^{-1} \n\mc{G}^*)^{-1} \f \mc{F}^* \m \mc{M}^{1/2}.
\end{eqnarray*}
 Therefore, $ \mc{A}^\dg_{\mc{M},\mc{N}} = \mc{N}^{-1} \n \mc{G}^* \f  (\mc{F}^* \m\mc{M} \m \mc{A} \n \mc{N}^{-1} \n\mc{G}^*)^{-1}\f \mc{F}^* \m \mc{M} $.
\end{proof}

In particular when the arbitrary-order tensor, $ \mc{A} $ is either left invertible or  right invertible, we have the following results.

\begin{corollary}\label{313}
Let a tensor $ \mc{A} \in \mathbb{C}^{I_1 \times \cdots \times I_M \times J_1 \times \cdots \times J_N} $ has the full rank decomposition. 
\begin{enumerate}
\item[(a)] If the tensor $\mc{A} $ is left invertible, then $\mc{A}^\dg_{\mc{M},\mc{N}} = \mc{N}^{-1}\n (\mc{A}^*\m \mc{M}\m \mc{A}\n \mc{N}^{-1})^{-1}\n \mc{A}^* \m\mc{M} $;
\item[(b)] If the tensor $\mc{A} $ is right invertible, then 
$ \mc{A}^\dg_{\mc{M},\mc{N}} = \mc{N}^{-1} \n \mc{A}^* \m (\mc{M}\m \mc{A}\n \mc{N}^{-1}\n \mc{A}^*)^{-1} \m \mc{M} $.
\end{enumerate}

\end{corollary}

  It is easy to see that the full rank factorizations of a tensor $\mc{A} \in  \mathbb{C}^{I_1 \times \cdots \times I_M \times J_1 \times \cdots \times J_N}$ are not unique: if $ \mc{A} = \mc{F} *_R\mc{G}$ is one full rank factorization, where $ \mc{F} \in \mathbb{C}^{I_1 \times \cdots \times I_M \times H_1 \times \cdots \times H_R} $ is the left invertible tensor and $ \mc{G} \in \mathbb{C}^{H_1 \times \cdots \times H_R \times J_1 \times \cdots \times J_N} $ is the right invertible tensor, then there exist a invertible  tensor $\mc{P}$ of appropriate size, such that, $\mc{A} = (\mc{F} *_R\mc{P})*_R(\mc{P}^{-1}*_R\mc{G})$ is another full rank factorization. 
  The following Theorem represents the result.

\begin{theorem}
Let $ \mc{A} \in \mathbb{C}^{I_1 \times \cdots \times I_M \times J_1 \times \cdots \times J_N} $ with $ rshrank(\mc{A}) = r = H_1 H_2  \cdots H_R $.
Then $ \mc{A} $ has infinitely many full rank decompositions. However if $\mc{A}$ has two full rank decompositions, as follows,
\begin{equation*}
\mc{A} = \mc{F}\f \mc{G} = \mc{F}_1 \f \mc{G}_1,
\end{equation*}
 where $ \mc{F},~\mc{F}_1 \in \mathbb{C}^{I_1 \times \cdots \times I_M \times H_1 \times \cdots \times H_R} $ and $ \mc{G},~\mc{G}_1 \in \mathbb{C}^{H_1 \times \cdots \times H_R \times J_1 \times \cdots \times J_N} $, then there exists an invertible tensor $ \mc{B} $ such that 
 \begin{equation*}
\mc{F}_1 = \mc{F}\f \mc{B}~~~and~~~ \mc{G}_1 = \mc{B}^{-1} \f \mc{G}. 
 \end{equation*}
 Moreover, 
 \begin{equation*}
\mc{F}_1 ^\dg = (\mc{F} \f \mc{B})^\dg = \mc{B}^{-1} \f \mc{F}^\dg ~~ and ~~ \mc{G}_1 ^\dg = (\mc{B}^{-1} \f \mc{G})^\dg = \mc{G}^\dg \f \mc{B}.
 \end{equation*}
    \end{theorem}

\begin{proof}
Suppose the tensor, $ \mc{A} \in \mathbb{C}^{I_1 \times \cdots \times I_M \times J_1 \times \cdots \times J_N} $  has two full rank decompositions, as follows, \begin{eqnarray}\label{frd11}
\mc{A} = \mc{F}\f \mc{G} =  \mc{F}_1 \f \mc{G}_1
\end{eqnarray}
where $\mc{F}, ~\mc{F}_1 \in \mathbb{C}^{I_1 \times \cdots \times I_M \times H_1 \times \cdots \times H_R}$ and $\mc{G}, ~\mc{G}_1 \in \mathbb{C}^{H_1 \times \cdots \times H_R \times J_1 \times \cdots \times J_N} $. Then 
\begin{equation*}
\mc{F} \f \mc{G}\n \mc{G}_1 ^\dg = \mc{F}_1 \f \mc{G}_1 \n \mc{G}_1 ^\dg.
\end{equation*} 
Substituting $ \mc{M}= \mc{I}_M $ and $\mc{N} = \mc{I}_N $ in Corollary \ref{313}(b) we have,  $ \mc{G}_1 \n \mc{G}_1 ^\dg = \mc{I}_R$. \\ Therefore, $\mc{F}_1 = \mc{F}\f (\mc{G}\n\mc{G}_1 ^\dg)$,  similarly we can find $\mc{G}_1 = (\mc{F}_1^\dg\m\mc{F})\f\mc{G}.$ \\ 
Let $ rsh(\mc{G}) = G =reshape(\mc{G},r,J_1\cdots J_N)$ and $ rsh(\mc{G}_1) = G_1 =reshape(\mc{G}_1,r,J_1\cdots J_N)$. Then $ rsh(\mc{G}\n \mc{G}_1 ^\dg) = G G_1^\dg \in \mathbb{C}^{r \times r}$ and
\begin{equation*}
r = rshrank(\mc{F}_1) = rshrank(\mc{F}\f(\mc{G}\n\mc{G}_1 ^\dg)) \leq rshrank(\mc{G}\n\mc{G}_1 ^\dg) = rank(GG_1^\dg) \leq r
\end{equation*}
Hence $ GG_1^\dg $  is invertible as it has full rank. This concluded $ \mc{G}\n\mc{G}_1^\dg = rsh^{-1}(GG_1^\dg) $ is invertible. Similarly $ \mc{F}_1^\dg\m\mc{F} $ is also invertible. Let $ \mc{B} = \mc{G}\n\mc{G}_1^\dg $ and $ \mc{C} = \mc{F}_1 ^\dg\m \mc{F}$.  Then 
\begin{equation*}
\mc{C}\f \mc{B} = \mc{F}_1^\dg\m\mc{F}\f\mc{G}\n \mc{G}_1 ^\dg = \mc{F}_1^\dg\m\mc{F}_1\f\mc{G}_1\n \mc{G}_1 ^\dg = \mc{I}_R
\end{equation*}
is equivalent to $ \mc{C} = \mc{B}^{-1} $. Therefore 
\begin{equation*}
 \mc{F}_1 = \mc{F}\f \mc{B}  \textnormal{~~and~~}  \mc{G}_1 = \mc{B}^{-1}\f \mc{G}.  
\end{equation*}
Further 
\begin{eqnarray*}
\mc{F}_1^\dg &=& ((\mc{F}\f\mc{B})^* \m\mc{F}\f\mc{B})^{-1}\f(\mc{F}\f\mc{B})^*\\
&=& \mc{B}^{-1}\f(\mc{F}^*\m\mc{F})^{-1}\f(\mc{B}^*)^{-1}\f \mc{B}^* \f\mc{F}^* \\
&=& \mc{B}^{-1}\f(\mc{F}^*\m\mc{F})^{-1}\f\mc{F}^* = \mc{B}^{-1}\f\mc{F}^\dg.
\end{eqnarray*}
Similarly $ \mc{G}_1 ^\dg = \mc{G}^\dg\f\mc{B}$. 
\end{proof}

\section{Reverse order law}

In this section, we present various necessary and sufficient conditions of the reverse-order law for the weighted Moore-Penrose inverses of tensors. The first result obtained below addresses the sufficient condition for reverse-order law of tensor.

\begin{theorem}
Let $\mc{A}\in \mathbb{C}^{I_1\times\cdots\times I_M \times J_1
\times\cdots\times J_N }$ and $\mc{B}\in \mathbb{C}^{J_1\times\cdots\times J_N \times K_1\times\cdots\times K_L }$. Let $\mc{M}  \in \mathbb{C}^{I_1\times\cdots\times I_M \times I_1 \times\cdots\times I_M},$ and $\mc{N} \in \mathbb{C}^{K_1\times\cdots\times K_L \times K_1 \times\cdots\times K_L}$ be a pair of Hermitian positive definite tensors.
If $ \mathfrak{R}(\mc{B}) = \mathfrak{R}(\mc{A}^*) $, Then 
\begin{eqnarray*}
(\mc{A}\n\mc{B})^\dg_{\mc{M},\mc{N}} = \mc{B}^\dg_{\mc{I}_N,\mc{N}}\n \mc{A}^\dg_{\mc{M},\mc{I}_N}.
\end{eqnarray*}

\end{theorem}

\begin{proof}
Let $\mc{X} =\mc{A}^\dag *_M\mc{A}*_N\mc{B} $ and $\mc{Y} =\mc{A}*_N\mc{B}*_L\mc{B}^\dag$. By using Lemma \ref{ID}(c), we get
\begin{eqnarray*}
 &&\mathfrak{R}[(\mc{X}^\dag)^*] 
 = \mathfrak{R}[(\mc{A}^\dag *_M\mc{A})^* *_N\mc{B}] \subseteq 
 \mathfrak{R}(\mc{A}^*)\\ 
 &&\hspace{3cm}
 \textnormal{~~and~~}\mathfrak{R}(\mc{Y}^\dag) 
 =\mathfrak{R}(\mc{B}*_L\mc{B}^\dag *_N\mc{A}^*) \subseteq \mathfrak{R}(\mc{B}). 
\end{eqnarray*}
 Similarly, from Eq. \eqref{lemma37} we have 
 \begin{eqnarray*}
 \mathfrak{R}(\mc{X}^*) =\mathfrak{R}(\mc{B}^* *_N\mc{A}^* *_M(\mc{A}^*)^\dag) =\mathfrak{R}[(\mc{A}*_N\mc{B})^*] ~~\textnormal{
 and}~~  \mathfrak{R}(\mc{Y}) =\mathfrak{R}(\mc{A}*_N\mc{B}).
 \end{eqnarray*}
 Further, from Lemma \ref{ID}[(a), (b)] and Lemma \ref{IDR1}[(a), (b)], we obtain 
 \begin{eqnarray*}
 \mc{X}^\dag *_N\mc{Y}^\dag =  \mc{X}^\dag *_N\mc{X}*_L\mc{B}^\dag *_N\mc{Y}^\dag =  (\mc{A}*_N\mc{B})^\dag *_M\mc{Y}*_N\mc{Y}^\dag = (\mc{A}*_N\mc{B})^\dag,
 \end{eqnarray*}
 i.e., \begin{equation}\label{3.8.1}
     (\mc{A}*_N\mc{B})^\dag = (\mc{A}^\dg\m\mc{A}\n\mc{B})^\dg\n(\mc{A}\n\mc{B}\1\mc{B}^\dg)^\dg.
 \end{equation}
 Let $\mc{A}_1 = \mc{M}^{1/2}\m\mc{A} $ and $ \mc{B}_1 = \mc{B}\1\mc{N}^{-1/2} $.
 Using the Lemma \ref{42}[(a),(b)], we get 
  \begin{eqnarray*}
 \mc{X} =\mc{A}_1^\dag *_M\mc{A}_1*_N\mc{B} ~~\textnormal{and}~~ \mc{Y} =\mc{A}*_N\mc{B}_1*_L\mc{B}_1^\dag.
  \end{eqnarray*}
 Now,  replacing $ \mc{A}$ and $\mc{B} $ by $ \mc{A}_1 $  and $\mc{B}_1 $  respectively on Eq.(\ref{3.8.1}), we get
   \begin{eqnarray*}
  (\mc{A}_1\n\mc{B}_1)^\dg 
 = (\mc{X}\1 \mc{N}^{-1/2})^\dg \n (\mc{M}^{1/2}\m \mc{Y})^\dg.
 \end{eqnarray*}
 Thus from corollary~\ref{3.4}, we can conclude 
 \begin{eqnarray*}
 (\mc{A}\n\mc{B})^\dg_{\mc{M},\mc{N}} = \mc{N}^{-1/2}\1 (\mc{A}_1\n\mc{B}_1)^\dg\m \mc{M}^{1/2} = \mc{X}^\dg_{\mc{I}_N,\mc{N}}\n \mc{Y}^\dg_{\mc{M},\mc{I}_N}.
 \end{eqnarray*}
 From the given condition and Lemma~\ref{ID}[(b),(c)], we have $ \mc{B}\1\mc{B}^\dg = \mc{A}^\dg\m \mc{A} $, i.e.,
 \begin{eqnarray*}
  \mc{A} = \mc{A}\n \mc{B}\1 \mc{B}^\dg = \mc{Y} ~~\textnormal{and}~~  \mc{B} = \mc{A}^\dg \m\mc{A}\n\mc{B} = \mc{X}.
 \end{eqnarray*}
 Hence, $ (\mc{A}\n\mc{B})^\dg_{\mc{M},\mc{N}} = \mc{B}^\dg_{\mc{I}_N,\mc{N}}\n \mc{A}^\dg_{\mc{M},\mc{I}_N} $.
 \end{proof}

Further, using Theorem~3.30 in  \cite{PanRad18} one can write a necessary and sufficient condition for reverse order law for arbitrary-order tensors, i.e., for $\mc{A}\in \mathbb{C}^{I_1\times\cdots\times I_M \times J_1\times\cdots\times J_N }$ and $\mc{B}\in \mathbb{C}^{J_1\times\cdots\times J_N \times K_1\times\cdots\times K_L}$. Then $(\mc{A}\n\mc{B})^{\dg} = \mc{B}^{\dg} \n \mc{A}^{\dg}$ if and only if
\begin{equation*}\label{eq19}
\mc{A}^{\dg}\m\mc{A}\n\mc{B}\1\mc{B}^* \n\mc{A}^* = \mc{B}\1\mc{B}^* \n \mc{A}^*, 
~~and
~~
\mc{B}\1\mc{B}^{\dg}\n\mc{A}^*\m\mc{A}\n\mc{B} = \mc{A}^* \m \mc{A} \n \mc{B}.
\end{equation*}
Now, utilizing the above result and the fact of Lemma \ref{ID}[(a),(c)], we conclude a beautiful result for necessary and sufficient condition for Moore-Penrose inverse of arbitrary-order tensor, as follows.

\begin{lemma}\label{RV1}
Let $\mc{A}\in \mathbb{C}^{I_1\times\cdots\times I_M \times J_1\times\cdots\times J_N }$ and $\mc{B}\in \mathbb{C}^{J_1\times\cdots\times J_N \times K_1\times\cdots\times K_L}$.
The Reverse order law hold for Moore-Penrose inverse, i.e.,$ (\mc{A}*_N\mc{B})^\dag = \mc{B}^\dag *_N\mc{A}^\dag $ 
if and only if~ \\
$ \mathfrak{R}(\mc{A}^* *_M\mc{A}*_N\mc{B})\subseteq \mathfrak{R}(\mc{B}) $ and $  \mathfrak{R}(\mc{B}*_L\mc{B}^* *_N\mc{A}^*) \subseteq \mathfrak{R}(\mc{A}^*)$.
\end{lemma}

The primary result of this paper is presented next under the impression of the properties of range space of arbitrary-order tensor. 

\begin{theorem}\label{RV2}
Let $\mc{A} \in \mathbb{C}^{I_1\times\cdots\times I_M \times J_1
\times\cdots\times J_N}, \mc{B} \in \mathbb{C}^{J_1\times\cdots\times J_N \times K_1
\times\cdots\times K_L}$. Let $ \mc{M} \in \mathbb{C}^{I_1\times\cdots\times I_M \times I_1 \times\cdots\times I_M}$, $\mc{N} \in \mathbb{C}^{K_1\times\cdots\times K_L \times K_1 \times\cdots\times K_L} $ and $ \mc{P} \in \mathbb{C}^{J_1\times\cdots\times J_N \times J_1 \times\cdots\times J_N} $
 are three Hermitian positive definite tensors. Then 
 \begin{equation*}
      (\mc{A}*_N\mc{B})_{\mc{M},\mc{N}} ^\dag = \mc{B}_{\mc{P},\mc{N}}^\dag *_N\mc{A}_{\mc{M},\mc{P}}  ^\dg
 \end{equation*}
     if and only if 
\begin{equation*}
\mathfrak{R}(\mc{A}^{\#}_{\mc{P},\mc{M}}*_M\mc{A}*_N\mc{B}) \subseteq \mathfrak{R}(\mc{B}) \textnormal{~~ and ~~} \mathfrak{R}(\mc{B}*_L\mc{B}^{\#}_{\mc{N},\mc{P}}*_N\mc{A}^{\#}_{\mc{P},\mc{M}}) \subseteq \mathfrak{R}(\mc{A}^{\#}_{\mc{P},\mc{M}}).
\end{equation*} 
\end{theorem}

\begin{proof}
From equation \eqref{2.17}, we have 
$(\mc{A}*_N\mc{B})_{\mc{M},\mc{N}} ^\dag = \mc{B}_{\mc{P},\mc{N}}^\dag *_N\mc{A}_{\mc{M},\mc{P}}^\dag $ if and only if 
\begin{eqnarray*}
 &&\mc{N}^{-1/2}*_L(\mc{M}^{1/2}*_M\mc{A}*_N\mc{B}*_L\mc{N}^{-1/2})^\dag *_M\mc{M}^{1/2} \\
&&= \mc{N}^{-1/2}*_L(\mc{P}^{1/2}*_N\mc{B}*_L\mc{N}^{-1/2})^\dag *_N\mc{P}^{1/2} *_N\mc{P}^{-1/2}*_N(\mc{M}^{1/2}*_M\mc{A}*_N\mc{P}^{-1/2})^\dag *_M\mc{M}^{1/2},
\end{eqnarray*}
is equivalent to,  if and only if   
\begin{equation*}
(\tilde{\mc{A}}*_N\tilde{\mc{B}})^\dag = \tilde{\mc{B}}^\dag *_N\tilde{\mc{A}}^\dag,
\end{equation*}
where   $ \tilde{\mc{A}} = \mc{M}^{1/2}*_M\mc{A}*_N\mc{P}^{-1/2}  $ and $ \tilde{\mc{B}}= \mc{P}^{1/2}*_N\mc{B}*_L\mc{N}^{-1/2} $. From Lemma \ref{RV1}, we have 
\begin{equation*}
(\mc{A}*_N\mc{B})_{\mc{M},\mc{N}} ^\dag = \mc{B}_{\mc{P},\mc{N}}^\dag *_N\mc{A}_{\mc{M},\mc{P}}^\dag
\end{equation*}
if and only if 
\begin{eqnarray}\label{29}
\mathfrak{R}(\tilde{\mc{A}}^* *_M\tilde{\mc{A}}*_N\tilde{\mc{B}}) \subseteq \mathfrak{R}(\tilde{\mc{B}}) \textnormal{~~and~~} \mathfrak{R}(\tilde{\mc{B}}*_L\tilde{\mc{B}}^* *_N\tilde{\mc{A}}^*) \subseteq \mathfrak{R}(\tilde{\mc{A}}^*),
\end{eqnarray}
which equivalently if and only if
\begin{eqnarray*}\label{rv1}
 \mathfrak{R}(\mc{P}^{1/2}*_N\mc{A}^{\#}_{\mc{P},\mc{M}}*_M\mc{A}*_N\mc{B}*_L\mc{N}^{-1/2}) \subseteq \mathfrak{R}(\mc{P}^{1/2}*_N\mc{B}*_L\mc{N}^{-1/2})\\
~~ and~~ \mathfrak{R}(\mc{P}^{1/2}*_N\mc{B}*_L\mc{B}^{\#}_{\mc{N},\mc{P}}*_N\mc{A}^{\#}_{\mc{P},\mc{M}}*_M\mc{M}^{-1/2}) \subseteq \mathfrak{R}(\mc{P}^{1/2}*_N\mc{A}^{\#}_{\mc{P},\mc{M}}*_M\mc{M}^{-1/2}). 
\end{eqnarray*}
Hence, $ (\mc{A}*_N\mc{B})_{\mc{M},\mc{N}} ^\dag = \mc{B}_{\mc{P},\mc{N}}^\dag *_N\mc{A}_{\mc{M},\mc{P}}  ^\dg $ if and only if
\begin{eqnarray*}
\mathfrak{R}(\mc{A}^{\#}_{\mc{P},\mc{M}}*_M\mc{A}*_N\mc{B}) \subseteq \mathfrak{R}(\mc{B}) \textnormal{ ~~and~~}  \mathfrak{R}(\mc{B}*_L\mc{B}^{\#}_{\mc{N},\mc{P}}*_N\mc{A}^{\#}_{\mc{P},\mc{M}}) \subseteq \mathfrak{R}(\mc{A}^{\#}_{\mc{P},\mc{M}}).
\end{eqnarray*}
This completes the proof.
\end{proof}

As a corollary to Theorem \ref{RV2}, we present another reverse order law for the weighted Moore-Penrose inverse of arbitrary-order tensor.

\begin{corollary}
Let $\mc{A} \in \mathbb{C}^{I_1\times\cdots\times I_M \times J_1
\times\cdots\times J_N}, \mc{B} \in \mathbb{C}^{J_1\times\cdots\times J_N \times K_1
\times\cdots\times K_L}$. Let $ \mc{M} \in \mathbb{C}^{I_1\times\cdots\times I_M \times I_1 \times\cdots\times I_M}$, $\mc{N} \in \mathbb{C}^{K_1\times\cdots\times K_L \times K_1 \times\cdots\times K_L} $  and $ \mc{P} \in \mathbb{C}^{J_1\times\cdots\times J_N \times J_1 \times\cdots\times J_N} $
 are three Hermitian positive definite tensors. Then 
 \begin{equation*}
(\mc{A}*_N\mc{B})_{\mc{M},\mc{N}} ^\dag = \mc{B}_{\mc{P},\mc{N}}^\dag *_N\mc{A}_{\mc{M},\mc{P}}^\dag
 \end{equation*}
if and only if
\begin{eqnarray*}
 \mc{A}_{\mc{M},\mc{P}}^\dag *_M\mc{A}*_N\mc{B}*_L\mc{B}^{\#}_{\mc{N},\mc{P}}*_N\mc{A}^{\#}_{\mc{P},\mc{M}} = \mc{B}*_L\mc{B}^{\#}_{\mc{N},\mc{P}}*_N\mc{A}^{\#}_{\mc{P},\mc{M}}\\  \textnormal{and~~}
 \mc{B}*_L\mc{B}_{\mc{P},\mc{N}}^\dag *_N\mc{A}^{\#}_{\mc{P},\mc{M}}*_M\mc{A}*_N\mc{B} = \mc{A}^{\#}_{\mc{P},\mc{M}}*_M\mc{A}*_N\mc{B}
\end{eqnarray*}

\end{corollary}

\begin{proof}
From Theorem~\ref{RV2}, Eq.\eqref{29} and Lemma~\ref{ID}(a), we have 
$ (\mc{A}*_N\mc{B})_{\mc{M},\mc{N}} ^\dag = \mc{B}_{\mc{P},\mc{N}}^\dag *_N\mc{A}_{\mc{M},\mc{P}}^\dag  $ if and only if
\begin{eqnarray*}
(\mc{P}^{1/2}*_N\mc{B}*_L\mc{N}^{-1/2})*_L&&\hspace{-.3cm}(\mc{P}^{1/2}*_N\mc{B}*_L\mc{N}^{-1/2})^\dag*_N\mc{P}^{1/2}*_N\mc{A}^{\#}_{\mc{P},\mc{M}}*_M\mc{A}*_N\mc{B}*_L\mc{N}^{-1/2} \\
&=& \mc{P}^{1/2}*_N\mc{A}^{\#}_{\mc{P},\mc{M}}*_M\mc{A}*_N\mc{B}*_L\mc{N}^{-1/2}
\end{eqnarray*}
and

\begin{eqnarray*} (\mc{P}^{1/2}*_N\mc{A}^{\#}_{\mc{P},\mc{M}}*_M\mc{M}^{-1/2})*_N&&\hspace{-.3cm}(\mc{P}^{1/2}*_N\mc{A}^{\#}_{\mc{P},\mc{M}}*_M\mc{M}^{-1/2})^\dag *_N\mc{P}^{1/2}*_N\mc{B}*_L\mc{B}^{\#}_{\mc{N},\mc{P}}*_N\mc{A}^{\#}_{\mc{M},\mc
{P}}*_M\mc{M}^{-1/2}\\
&=& \mc{P}^{1/2}*_N\mc{B}*_L\mc{B}^{\#}_{\mc{N},\mc{P}}*_N\mc{A}^{\#}_{\mc{P},\mc{M}}*_M\mc{M}^{-1/2},
\end{eqnarray*}
\textnormal{i.e., if and only if}
\begin{eqnarray*}
&&\mc{B}*_L\mc{B}_{\mc{P},\mc{N}}^\dag *_N\mc{A}^{\#}_{\mc{P},\mc{M}}*_M\mc{A}*_N\mc{B} 
= \mc{A}^{\#}_{\mc{P},\mc{M}}*_M\mc{A}*_N\mc{B} ~~~\textnormal{and} \\
&&[(\mc{M}^{1/2}*_M\mc{A}*_N\mc{P}^{-1/2})^\dag *_M(\mc{M}^{1/2}*_M\mc{A}*_N\mc{P}^{-1/2})]^* *_N\mc{P}^{1/2}*_N\mc{B}*_L\mc{B}^{\#}_{\mc{N},\mc{P}}*_N\mc{A}^{\#}_{\mc{M},\mc{P}}*_M\mc{M}^{-1/2}\\
&&\hspace{2cm}= \mc{P}^{1/2}*_N\mc{B}*_L\mc{B}^{\#}_{\mc{N},\mc{P}}*_N\mc{A}^{\#}_{\mc{P},\mc{M}}*_M\mc{M}^{-1/2},
\end{eqnarray*}
i.e., if and only if 
\begin{eqnarray*}
&&\mc{B}*_L\mc{B}_{\mc{P},\mc{N}}^\dag *_N\mc{A}^{\#}_{\mc{P},\mc{M}}*_M\mc{A}*_N\mc{B} 
= \mc{A}^{\#}_{\mc{P},\mc{M}}*_M\mc{A}*_N\mc{B}\\
&& \textnormal{~~and~~}
\mc{A}_{\mc{M},\mc{P}}^\dag *_M\mc{A}*_N\mc{B}*_L\mc{B}^{\#}_{\mc{N},\mc{P}}*_N\mc{A}^{\#}_{\mc{P},\mc{M}} = \mc{B}*_L\mc{B}^{\#}_{\mc{N},\mc{P}}*_N\mc{A}^{\#}_{\mc{P},\mc{M}}.
\end{eqnarray*}
This completes the proof.
\end{proof}


\begin{theorem}
Let $ \mc{A} \in \mathbb{C}^{I_1\times\cdots\times I_M \times J_1
\times\cdots\times J_N}, \mc{B} \in \mathbb{C}^{J_1\times\cdots\times J_N \times K_1
\times\cdots\times K_L}$. Let $ \mc{M} \in \mathbb{C}^{I_1\times\cdots\times I_M \times I_1 \times\cdots\times I_M}$, $\mc{N} \in \mathbb{C}^{K_1\times\cdots\times K_L \times K_1 \times\cdots\times K_L} $ and $ \mc{P} \in \mathbb{C}^{J_1\times\cdots\times J_N \times J_1 \times\cdots\times J_N} $ are positive definite Hermitian tensors. Then 
\begin{equation*}
(\mc{A}\n\mc{B})^\dg_{\mc{M},\mc{N}} = \mc{B}^\dg_{\mc{P},\mc{N}}\n \mc{A}^\dg_{\mc{M},\mc{P}}
\end{equation*}
if and only if 
\begin{equation*}
(\mc{A}^\dg_{\mc{M},\mc{P}}\m \mc{A}\n\mc{B})^\dg_{\mc{P},\mc{N}} = \mc{B}^\dg_{\mc{P},\mc{N}}\n\mc{A}^\dg_{\mc{M},\mc{P}}\m\mc{A}
\textnormal{~~and~~} (\mc{A}\n\mc{B}\1\mc{B}^\dg_{\mc{P},\mc{N}})^\dg_{\mc{M},\mc{P}} = \mc{B}\1    \mc{B}^\dg_{\mc{P},\mc{N}}\n\mc{A}^\dg_{\mc{M},\mc{P}}.
\end{equation*}
\end{theorem}

\begin{proof}
Suppose, $ (\mc{A}\n\mc{B})^\dg_{\mc{M},\mc{N}} = \mc{B}^\dg_{\mc{P},\mc{N}}\n \mc{A}^\dg_{\mc{M},\mc{P}}.$ 
Now one can write
\begin{eqnarray*}
 (\mc{A}^\dg_{\mc{M},\mc{P}}\m \mc{A}\n\mc{B})\1(\mc{B}^\dg_{\mc{P},\mc{N}}\n\mc{A}^\dg_{\mc{M},\mc{P}}\m\mc{A}) \n(\mc{A}^\dg_{\mc{M},\mc{P}}\n \mc{A}\n\mc{B})
=   \mc{A}^\dg_{\mc{M},\mc{P}}\m\mc{A}\n\mc{B}.
\end{eqnarray*}
Further we can write
\begin{eqnarray*}
(\mc{B}^\dg_{\mc{P},\mc{N}}\n\mc{A}^\dg_{\mc{M},\mc{P}}\m\mc{A})\n(\mc{A}^\dg_{\mc{M},\mc{P}}\m \mc{A}\n\mc{B})\1(\mc{B}^\dg_{\mc{P},\mc{N}}\n\mc{A}^\dg_{\mc{M},\mc{P}}\m\mc{A})
= \mc{B}^\dg_{\mc{P},\mc{N}}\n\mc{A}^\dg_{\mc{M},\mc{P}}\m\mc{A}.
\end{eqnarray*}
\begin{eqnarray*}
\textnormal{Also~}, [\mc{P}&&\hspace{-.3cm}\n(\mc{A}^\dg_{\mc{M},\mc{P}}\m \mc{A}\n\mc{B})\1(\mc{B}^\dg_{\mc{P},\mc{N}}\n\mc{A}^\dg_{\mc{M},\mc{P}}\m\mc{A})]^* \\
&=&  [(\mc{P}\n\mc{A}^\dg_{\mc{M},\mc{P}}\m \mc{A})\n \mc{P}^{-1}\n(\mc{P}\n \mc{B}\1\mc{B}^\dg_{\mc{P},\mc{N}})\n\mc{P}^{-1}\n(\mc{P}\n\mc{A}^\dg_{\mc{M},\mc{P}}\m\mc{A})]^* \\
&=&  (\mc{P}\n\mc{A}^\dg_{\mc{M},\mc{P}}\m \mc{A})\n \mc{P}^{-1}\n(\mc{P}\n \mc{B}\1\mc{B}^\dg_{\mc{P},\mc{N}})\n\mc{P}^{-1}\n(\mc{P}\n\mc{A}^\dg_{\mc{M},\mc{P}}\m\mc{A})\\
&=&  \mc{P}\n(\mc{A}^\dg_{\mc{M},\mc{P}}\m \mc{A}\n\mc{B})\1(\mc{B}^\dg_{\mc{P},\mc{N}}\n\mc{A}^\dg_{\mc{M},\mc{P}}\m\mc{A}),
\end{eqnarray*}

\begin{eqnarray*}
\textnormal{and~~}  [\mc{N}\1(\mc{B}^\dg_{\mc{P},\mc{N}}\n\mc{A}^\dg_{\mc{M},\mc{P}}\m\mc{A})&&\hspace{-.3cm} \n(\mc{A}^\dg_{\mc{M},\mc{P}}\m \mc{A}\n\mc{B})]^*\\
&=& [\mc{N}\1(\mc{B}^\dg_{\mc{P},\mc{N}}\n\mc{A}^\dg_{\mc{M},\mc{P}})\m\mc{A}\n\mc{B})]^* \\
&=& \mc{N}\1(\mc{B}^\dg_{\mc{P},\mc{N}}\n\mc{A}^\dg_{\mc{M},\mc{P}}\m\mc{A})\n(\mc{A}^\dg_{\mc{M},\mc{P}}\m \mc{A}\n\mc{B}).
\end{eqnarray*}
Hence, 
\begin{equation*}
(\mc{A}^\dg_{\mc{M},\mc{P}}\m \mc{A}\n\mc{B})^\dg_{\mc{P},\mc{N}} = \mc{B}^\dg_{\mc{P},\mc{N}}\n\mc{A}^\dg_{\mc{M},\mc{P}}\m\mc{A}.
\end{equation*} 
By similar arguments one can also show that,$ (\mc{A}\n\mc{B}\1\mc{B}^\dg_{\mc{P},\mc{N}})^\dg_{\mc{M},\mc{P}} = \mc{B}\1\mc{B}^\dg_{\mc{P},\mc{N}}\n\mc{A}^\dg_{\mc{M},\mc{P}} $.\\
Conversely, For proving converse, first we prove a identity.\\
From Eq.\eqref{2.17} and Eq.\eqref{3.8.1} we have,
\begin{eqnarray*}
(\mc{A}\n\mc{B})^\dg_{\mc{M},\mc{N}}
&=& \mc{N}^{-1/2}\1[(\mc{M}^{1/2}\m\mc{A}\n\mc{P}^{-1/2})\n(\mc{P}^{1/2}\n\mc{B}\1\mc{N}^{-1/2}]^\dg\m\mc{M}^{1/2}\\
&=& \mc{N}^{-1/2}\1[(\mc{M}^{1/2}\m\mc{A}\n\mc{P}^{-1/2})^\dg\n(\mc{M}^{1/2}\m\mc{A}\n\mc{P}^{-1/2})\n(\mc{P}^{1/2}\n\mc{B}\1\mc{N}^{-1/2})]^\dg\\
&&\hspace{.5cm}\n[(\mc{M}^{1/2}\m\mc{A}\n\mc{P}^{-1/2})\n(\mc{P}^{1/2}\n\mc{B}\1\mc{N}^{-1/2})\n(\mc{P}^{1/2}\n\mc{B}\1\mc{N}^{-1/2})^\dg]^\dg \m\mc{M}^{1/2}\\
&=&  \mc{N}^{-1/2}\1[\mc{P}^{1/2}\n\mc{A}^\dg_{\mc{M},\mc{P}}\m \mc{A}\n\mc{B}\1\mc{N}^{-1/2}]^\dg \n \mc{P}^{1/2}\n \mc{P}^{-1/2}\n
\\
&&\hspace{1cm} 
[\mc{M}^{1/2}\m\mc{A}\n\mc{B}\1\mc{B}^\dg_{\mc{P},\mc{N}}\n \mc{P}^{-1/2}]^\dg\m\mc{M}^{1/2}\\
&=& (\mc{A}^\dg_{\mc{M},\mc{P}}\m \mc{A}\n\mc{B})^\dg_{\mc{P},\mc{N}}\n (\mc{A}\n\mc{B}\1\mc{B}^\dg_{\mc{P},\mc{N}})^\dg_{\mc{M},\mc{P}}.
\end{eqnarray*}
Further, using the given hypothesis and above identity, we can write
\begin{eqnarray*}
(\mc{A}\n\mc{B})^\dg_{\mc{M},\mc{N}}  
&=& (\mc{B}^\dg_{\mc{P},\mc{N}}\n\mc{A}^\dg_{\mc{M},\mc{P}}\m\mc{A})\n( \mc{B}\1\mc{B}^\dg_{\mc{P},\mc{N}}\n\mc{A}^\dg_{\mc{M},\mc{P}})\\
&=& (\mc{B}^\dg_{\mc{P},\mc{N}}\n\mc{A}^\dg_{\mc{M},\mc{P}}\m\mc{A})\n(\mc{A}^\dg_{\mc{M},\mc{P}}\m\mc{A}\n \mc{B})\1(\mc{B}^\dg_{\mc{P},\mc{N}}\n\mc{A}^\dg_{\mc{M},\mc{P}}\m\mc{A})\n\mc{A}^\dg_{\mc{M},\mc{P}}\\
&=& 
\mc{B}^\dg_{\mc{P},\mc{N}}\n\mc{A}^\dg_{\mc{M},\mc{P}}.
\end{eqnarray*}
This completes the proof.
\end{proof}

In the next theorem, we develop the characterization for the weighted Moore-Penrose inverse of the product of arbitrary-order tensors $\mc{A}$ and $\mc{B}$, as follows.

\begin{theorem}\label{4.6}
Let $\mc{A} \in \mathbb{C}^{I_1\times\cdots\times I_M \times J_1
\times\cdots\times J_N}, \mc{B} \in \mathbb{C}^{J_1\times\cdots\times J_N \times K_1
\times\cdots\times K_L}$. Let $ \mc{M} \in \mathbb{C}^{I_1\times\cdots\times I_M \times I_1 \times\cdots\times I_M}$, $\mc{N} \in \mathbb{C}^{K_1\times\cdots\times K_L \times K_1 \times\cdots\times K_L} $  and $ \mc{P} \in \mathbb{C}^{J_1\times\cdots\times J_N \times J_1 \times\cdots\times J_N} $
are three Hermitian positive definite tensors. Then 
\begin{equation*}
(\mc{A}*_N\mc{B})^\dag_{\mc{M},\mc{N}} = (\mc{B}_1)_{\mc{P},\mc{N}}^\dag *_N(\mc{A}_1)_{\mc{M},\mc{P}}^\dag,  
\end{equation*}
 where $ \mc{A}_1 = \mc{A}*_N\mc{B}_1*_L(\mc{B}_1)_{\mc{P},\mc{N}}^\dag $ and $ \mc{B}_1 =  \mc{A}_{\mc{M},\mc{P}}^\dag*_M\mc{A}*_N\mc{B}  $.
\end{theorem}

\begin{proof}
\begin{eqnarray}\label{4.5.1}
 \mc{A}*_N\mc{B} &=& \mc{A}*_N\mc{A}^\dag_{\mc{M},\mc{P}}*_M\mc{A}*_N\mc{B}
 = \mc{A}*_N\mc{B}_1\\ \nonumber
 &=& \mc{A}*_N\mc{B}_1*_L(\mc{B}_1)^\dag_{\mc{P},\mc{N}}*_N\mc{B}_1 = \mc{A}_1*_N\mc{B}_1.    
\end{eqnarray}
\begin{eqnarray}\label{4.5.2}\nonumber
    \mc{A}_{\mc{M},\mc{P}}^\dag*_M\mc{A}_1 
    &=& \mc{A}^\dag_{\mc{M},\mc{P}}*_M\mc{A}*_N\mc{A}^\dag_{\mc{M},\mc{P}}*_M\mc{A}*_N\mc{B}*_L(\mc{B}_1)^\dag_{\mc{P},\mc{N}}\\
    &=& \mc{B}_1*_L(\mc{B}_1)^\dag_{\mc{P},\mc{N}}.  
\end{eqnarray}
\begin{eqnarray}\label{4.5.3}\nonumber
    \mc{A}^\dag_{\mc{M},\mc{P}}*_M\mc{A}_1 &=& \mc{A}^\dag_{\mc{M},\mc{P}}*_M\mc{A}_1*_N(\mc{A}_1)^\dag_{\mc{M},\mc{P}}*_M\mc{A}_1\\\nonumber
    &=& \mc{B}_1*_L(\mc{B}_1)^\dag_{\mc{P},\mc{N}}*_N(\mc{A}_1)^\dag_{\mc{M},\mc{P}}*_M\mc{A}_1. 
\end{eqnarray}
From \eqref{4.5.2} and \eqref{4.5.3}, we have 
\begin{eqnarray*}
\mc{P}*_N\mc{B}_1*_L(\mc{B}_1)^\dag_{\mc{P},\mc{N}} 
= [\mc{P}*_N\mc{B}_1*_L(\mc{B}_1)^\dag_{\mc{P},\mc{N}}]*_N\mc{P}^{-1}*_N[\mc{P}*_N(\mc{A}_1)^\dag_{\mc{M},\mc{P}}*_M\mc{A}_1].
\end{eqnarray*}
Therefore, 
\begin{eqnarray*}
\mc{P}*_N\mc{B}_1*_L(\mc{B}_1)^\dag_{\mc{P},\mc{N}} 
&=& [\mc{P}*_N\mc{B}_1*_L(\mc{B}_1)^\dag_{\mc{P},\mc{N}}]^* \\
&=& \mc{P}*_N(\mc{A}_1)^\dag_{\mc{M},\mc{P}}*_M\mc{A}_1*_N\mc{B}_1*_L(\mc{B}_1)^\dag_{\mc{P},\mc{N}}\\
&=& \mc{P}*_N(\mc{A}_1)^\dag_{\mc{M},\mc{P}}*_M\mc{A}_1.
\end{eqnarray*}
Hence, 
\begin{eqnarray}\label{4.5.4}
 {\mc{B}_1*_L(\mc{B}_1)^\dag_{\mc{P},\mc{N}} 
 = (\mc{A}_1)^\dag_{\mc{M},\mc{P}}*_M\mc{A}_1 
 = \mc{A}^\dag_{\mc{M},\mc{P}}*_M\mc{A}_1}.
\end{eqnarray}
Let $ \mc{X} = \mc{A}*_N\mc{B} $ and $ \mc{Y} = (\mc{B}_1)^\dag_{\mc{P},\mc{N}}*_N(\mc{A}_1)^\dag_{\mc{M},\mc{P}} $.
Using \eqref{4.5.1} and \eqref{4.5.4} we obtain
\begin{eqnarray*}
\mc{X}*_L\mc{Y}*_M\mc{X} 
&=& \mc{A}*_N\mc{B}*_L(\mc{B}_1)^\dag_{\mc{P},\mc{N}}*_M(\mc{A}_1)^\dag_{\mc{M},\mc{P}}*_N\mc{A}_1*_N\mc{B}_1 \\
&=& \mc{A}_1*_N\mc{B}_1*_L(\mc{B}_1)^\dag_{\mc{P},\mc{N}}*_N\mc{B}_1*_L(\mc{B}_1)^\dag_{\mc{P},\mc{N}}*_N\mc{B}_1
= \mc{X},
\end{eqnarray*}
\begin{eqnarray*}
\mc{Y}*_M\mc{X}*_L\mc{Y} 
&=& (\mc{B}_1)^\dag_{\mc{P},\mc{N}}*_N\mc{B}_1*_L(\mc{B}_1)^\dag_{\mc{P},\mc{N}}*_N\mc{B}_1*_L(\mc{B}_1)^\dag_{\mc{P},\mc{N}}*_N(\mc{A}_1)^\dag_{\mc{M},\mc{P}}
= \mc{Y},
\end{eqnarray*}
\begin{eqnarray*}
\mc{M}*_M\mc{X}*_L\mc{Y} 
&=& \mc{M}*_M\mc{A}_1*_N(\mc{A}_1)^\dag_{\mc{M},\mc{P}}*_M\mc{A}_1*_N(\mc{A}_1)^\dag_{\mc{M},\mc{P}}
= (\mc{M}*_M\mc{X}*_L\mc{Y})^* 
\end{eqnarray*}
and 
\begin{eqnarray*}
\mc{N}*_L\mc{Y}*_M\mc{X} 
&=& \mc{N}*_L(\mc{B}_1)^\dag_{\mc{P},\mc{N}}*_N\mc{B}_1*_L(\mc{B}_1)^\dag_{\mc{P},\mc{N}}*_N\mc{B}_1 
= (\mc{N}*_L\mc{Y}*_M\mc{X})^*.
\end{eqnarray*}
Hence, $ \mc{X}^\dag_{\mc{M},\mc{N}} = \mc{Y} $, i.e., $ (\mc{A}*_N\mc{B})^\dag_{\mc{M},\mc{N}} = (\mc{B}_1)^\dag_{\mc{P},\mc{N}}*_N(\mc{A}_1)^\dag_{\mc{M},\mc{P}} $.
\end{proof}
We shall present the following example as a confirmation of the above Theorem.
\begin{example}
Let $\mc{A}_1 = \mc{A}*_2\mc{B}_1*_1(\mc{B}_1)_{\mc{P},\mc{N}}^\dag $ and $ \mc{B}_1 =  \mc{A}_{\mc{M},\mc{P}}^\dag*_1\mc{A}*_2\mc{B}$, where 
$~\mc{A}=(a_{ijk})
 \in \mathbb{R}^{3\times2\times4},~~\mc{B}=(b_{ijk})
 \in \mathbb{R}^{2\times4\times3}, ~~\mc{M}=(m_{ij})
 \in \mathbb{R}^{3\times3}, ~~ \mc{N}=(n_{ij})
 \in \mathbb{R}^{3\times3}$ and $\mc{P}=(p_{ijkl})
 \in \mathbb{R}^{2\times4\times2\times4}$ such that

\begin{eqnarray*}
a_{ij1} =
\begin{pmatrix}
-1 & 2 \\
1 & -1 \\
0 & 1 \\
    \end{pmatrix},
a_{ij2} =
    \begin{pmatrix}
1 & 0 \\
0 & 0 \\
1 & 0 \\
\end{pmatrix},
a_{ij3} =
\begin{pmatrix}
2 & 0 \\
1 & 1 \\
0 & 0 \\
    \end{pmatrix},
a_{ij4} =
    \begin{pmatrix}
3 & 2 \\
1 & -1 \\
0 & 1 \\
\end{pmatrix},
\end{eqnarray*}
\begin{eqnarray*}
b_{ij1} = 
    \begin{pmatrix}
-1 & 2 & 1 & 1\\
0 & 1 & 1 & 0\\
    \end{pmatrix},
b_{ij2} =
    \begin{pmatrix}
0 & 1 & 1 & 1\\
1 & 1 & 0 & 1\\
    \end{pmatrix},
    b_{ij3} =
    \begin{pmatrix}
0 & 1 & 1 & 1\\
1 & 1 & 0 & 1\\
    \end{pmatrix},
       \end{eqnarray*}  

    \begin{eqnarray*}
M =
\begin{pmatrix}
3 & 0 & 1 \\
0 & 2 & 0 \\
1 & 0 & 2 \\
    \end{pmatrix},
N =
    \begin{pmatrix}
1 & 1 & 0 \\
1 & 2 & 0 \\
0 & 0 & 1 \\
\end{pmatrix},
\end{eqnarray*}

  \begin{eqnarray*}
p_{ij11} =
\begin{pmatrix}
1 & 0 & 0 & 1\\
0 & 0 & 0 & 0\\
    \end{pmatrix},
p_{ij12} =
    \begin{pmatrix}
0 & 1 & 0 & 0\\
0 & 0 & 0 & 0\\
\end{pmatrix}
p_{ij13} =
\begin{pmatrix}
0 & 0 & 1 & 0\\
1 & 0 & 0 & 0\\
    \end{pmatrix},
p_{ij14} =
    \begin{pmatrix}
1 & 0 & 0 & 3\\
0 & 0 & 1 & 0\\
\end{pmatrix}
\end{eqnarray*}  
     \begin{eqnarray*}
p_{ij21} =
\begin{pmatrix}
0 & 0 & 1 & 0\\
2 & 0 & 0 & 0\\
    \end{pmatrix},
p_{ij22} =
    \begin{pmatrix}
0 & 0 & 0 & 0\\
0 & 2 & 2 & 1\\
\end{pmatrix}
p_{ij23} =
\begin{pmatrix}
0 & 0 & 0 & 1\\
0 & 2 & 5 & 0\\
    \end{pmatrix},
p_{ij24} =
    \begin{pmatrix}
0 & 0 & 0 & 0\\
0 & 1 & 0 & 1\\
\end{pmatrix}
\end{eqnarray*}  
 
 Then $\mc{A}_1 = (\tilde{a}_{ijk})
 \in \mathbb{R}^{3\times2\times4}$, $\mc{B}_1 = (\tilde{b}_{ijk})
 \in \mathbb{R}^{2\times4\times3},~~ (\mc{A}_1)^\dag_{\mc{M},\mc{P}}  = (x_{ijk}) \in \mathbb{R}^{2\times4\times3}$\\  and $(\mc{B}_1)^\dag_{\mc{P},\mc{N}} =  (y_{ijk})  \in \mathbb{R}^{3\times2\times4}$  such that 
\begin{eqnarray*}
\tilde{a}_{ij1} =
\begin{pmatrix}
-1 & 2 \\
0 & -1 \\
0 & 1 \\
    \end{pmatrix},
\tilde{a}_{ij2} =
    \begin{pmatrix}
1 & 0 \\
0 & 1 \\
1 & 0 \\
\end{pmatrix},
\tilde{a}_{ij3} =
\begin{pmatrix}
2 & 0 \\
1 & 1 \\
0 & 0 \\
    \end{pmatrix},
\tilde{a}_{ij4} =
    \begin{pmatrix}
3 & 2 \\
1 & -1 \\
0 & 1 \\
\end{pmatrix},
\end{eqnarray*}
\begin{eqnarray*}
\tilde{b}_{ij1} = 
    \begin{pmatrix}
-0.3450  &  1.0728  &  0.7438  &  0.4134\\
-0.2067  & 1.1965  &  -0.4265  & -0.8661\\
    \end{pmatrix},
\tilde{b}_{ij2} =
    \begin{pmatrix}
-0.3319  &  1.7409  &  1.0320  &  0.4483\\
-0.2242 &   1.1004  & -0.3217  & -0.5167
    \end{pmatrix},
\end{eqnarray*}  
 \begin{eqnarray*}
   \tilde{b}_{ij3} =
    \begin{pmatrix}
 1.5109  &  4.0568  &  0.2402  & -0.6376 \\
 0.3188 &   1.2533  &  0.0873  & -0.3755
    \end{pmatrix},
    x_{ij1} = 
    \begin{pmatrix}
 -0.2052  & -0.1339  &  0.1514 &   0.1194\\
 -0.0597  & -0.1616   & 0.0247    & 0.1936
    \end{pmatrix},
       \end{eqnarray*}  
 \begin{eqnarray*}
x_{ij2} =
    \begin{pmatrix}
   0.0218  &  0.4469   & 0.1470  &  0.0582\\
   -0.0291  &  0.5066 & -0.1587  & -0.4178

    \end{pmatrix},
x_{ij3} =
    \begin{pmatrix}
 0.4236  &  0.9360  & -0.0146 &  -0.2038\\
0.1019   & 0.2271    & 0.0553 &  -0.0378
    \end{pmatrix},
       \end{eqnarray*}  
 
 \begin{eqnarray*}
y_{ij1} =
\begin{pmatrix}
 0.4783  & -1.6522\\
   -0.5217 &   1.3478\\
    0.1304  & -0.0870
    \end{pmatrix},
y_{ij2} =
    \begin{pmatrix}
-0.5217  &  0.6522\\
    0.4783 &  -0.3478\\
    0.1304  &  0.0870
\end{pmatrix},
y_{ij3} =
\begin{pmatrix}
 -0.3043 &    0.6522\\
    0.6957 &  -0.3478\\
   -0.1739  &  0.0870
    \end{pmatrix},
    \end{eqnarray*}
 \begin{eqnarray*}
y_{ij4} =
    \begin{pmatrix}
   -0.7826  & -1.6522\\
    1.2174  &  1.3478\\
   -0.3043  & -0.0870
\end{pmatrix},
\end{eqnarray*}
 Thus
 \begin{eqnarray*}
 (\mc{A}*_N\mc{B})^\dag_{\mc{M},\mc{N}} = 
 \begin{pmatrix}
-0.4783  &  0.6522 &  -0.0435\\
    0.5217  & -0.3478 &  -0.0435\\
   -0.1304 &   0.0870  &  0.2609
\end{pmatrix}
= (\mc{B}_1)^\dag_{\mc{P},\mc{N}}*_N(\mc{A}_1)^\dag_{\mc{M},\mc{P}}. 
 \end{eqnarray*}
Hence $ (\mc{A}*_N\mc{B})^\dag_{\mc{M},\mc{N}} = (\mc{B}_1)^\dag_{\mc{P},\mc{N}}*_N(\mc{A}_1)^\dag_{\mc{M},\mc{P}}$
\end{example}

Further, using the Lemma 4 in \cite{weit2} on an arbitrary-order tensor $\mc{A} \in \mathbb{C}^{I_1\times\cdots\times I_M \times J_1 \times\cdots\times J_N}$ with Hermitian positive definite tensors $\mc{M} \in \mathbb{C}^{I_1\times\cdots\times I_M \times I_1 \times\cdots\times I_M}$ and $\mc{N} \in \mathbb{C}^{J_1\times\cdots\times J_N \times J_1 \times\cdots\times J_N} $ one can write the following identity
\begin{equation}\label{12.1}
\mathfrak{R}(\mc{A}^\dag_{\mc{M},\mc{N}}*_M\mc{A}) = \mathfrak{R}(\mc{A}^\#_{\mc{N},\mc{M}})
\end{equation}
Using the above identity, a sufficient condition for the reverse order law for weighted Moore-Penrose inverse of tensor is presented next.

\begin{corollary}
Let $ \mc{A} \in \mathbb{C}^{I_1\times\cdots\times I_M \times J_1
\times\cdots\times J_N}, \mc{B} \in \mathbb{C}^{J_1\times\cdots\times J_N \times K_1
\times\cdots\times K_L}$. Let $ \mc{M} \in \mathbb{C}^{I_1\times\cdots\times I_M \times I_1 \times\cdots\times I_M}$, $\mc{N} \in \mathbb{C}^{K_1\times\cdots\times K_L \times K_1 \times\cdots\times K_L} $  and $ \mc{P} \in \mathbb{C}^{J_1\times\cdots\times J_N \times J_1 \times\cdots\times J_N} $ are positive definite Hermitian tensors. If
\begin{equation*}
\mathfrak{R}(\mc{B}) \subseteq \mathfrak{R}(\mc{A}^{\#}_{\mc{P},\mc{M}}) ~~~ and~~~  \mc{N}(\mc{B}^{\#}_{\mc{N}^{1/2},\mc{P}^{1/2}}) \subseteq \mc{N}(\mc{A}),
\end{equation*}
Then
\begin{equation*}
(\mc{A}*_N\mc{B})^\dag_{\mc{M},\mc{N}} = \mc{B}^\dag_{\mc{P},\mc{N}}*_N\mc{A}^\dag_{\mc{M},\mc{P}}.
\end{equation*}
\end{corollary}

\begin{proof}
From Theorem\ref{4.6} we have, $ (\mc{A}*_N\mc{B})^\dag_{\mc{M},\mc{N}} = (\mc{B}_1)_{\mc{P},\mc{N}}^\dag *_N(\mc{A}_1)_{\mc{M},\mc{P}}^\dag $, where \\
$ \mc{A}_1 = \mc{A}*_N\mc{B}_1*_L(\mc{B}_1)_{\mc{P},\mc{N}}^\dag $ and $ \mc{B}_1 =  \mc{A}_{\mc{M},\mc{P}}^\dag*_M\mc{A}*_N\mc{B}  $.\\
From Eq.\eqref{12.1} and given hypothesis, we have
\begin{eqnarray*}
\mathfrak{R}(\mc{A}^\dag_{\mc{M},\mc{P}}*_M\mc{A}) = \mathfrak{R}(\mc{A}^{\#}_{\mc{P},\mc{M}}) \supseteq \mathfrak{R}(\mc{B})
\end{eqnarray*}
So there exists $ \mc{P} \in \mathbb{C}^{J_1 \times \cdots \times J_N \times K_1 \times \cdots \times K_L} $ such that $ \mc{B} = \mc{A}^\dag_{\mc{M},\mc{N}}*_M\mc{A}*_N\mc{P}$.
 Now, 
 \begin{eqnarray*}
 \mc{B}_1 = \mc{A}^\dag_{\mc{M},\mc{N}}*_M\mc{A}*_N\mc{B} 
 = \mc{A}^\dag_{\mc{M},\mc{N}}*_M\mc{A}*_N\mc{P} = \mc{B}.
 \end{eqnarray*}
Hence $ \mc{A}_1 = \mc{A}*_N\mc{B}*_L\mc{B}^\dag_{\mc{N},\mc{P}} $.\\ Further, 
we have, $ \mc{N}(\mc{B}^{\#}_{\mc{N}^{1/2},\mc{P}^{1/2}}) \subseteq \mc{N}(\mc{A}) $, which is equivalent to 
\begin{equation*}
\mathfrak{R}(\mc{P}^{-1/2}*_N\mc{A}^*) = \mathfrak{R}(\mc{A}^*) \subseteq \mathfrak{R}[(\mc{B}^{\#}_{\mc{N}^{1/2},\mc{P}^{1/2}})^*] = \mathfrak{R}[(\mc{N}^{-1/2}*_L\mc{B}^* *_N\mc{P}^{1/2})^*] 
\end{equation*}
Then from Lemma\ref{ID}(a), we have 
\begin{equation*}
(\mc{A}*_N\mc{P}^{-1/2})*_N(\mc{N}^{-1/2}*_L\mc{B}^* *_N\mc{P}^{1/2})^\dag*_L(\mc{N}^{-1/2}*_L\mc{B}^* *_N\mc{P}^{1/2}) = \mc{A}*_N\mc{P}^{-1/2}, 
\end{equation*}
which equivalently
\begin{equation*}
(\mc{A}*_N\mc{P}^{-1/2})*_N[(\mc{P}^{1/2}*_N\mc{B}*_L\mc{N}^{-1/2})*_L(\mc{P}^{1/2}*_N\mc{B}*_L\mc{N}^{-1/2})^\dag]^* = \mc{A}*_N\mc{P}^{-1/2},
\end{equation*}
that is 
\begin{equation*}
    \mc{A}*_N\mc{B}*_L\mc{N}^{-1/2}*_L(\mc{P}^{1/2}*_N\mc{B}*_L\mc{N}^{-1/2})^\dag*_N\mc{P}^{1/2} = \mc{A},
\end{equation*}
i.e.
\begin{equation*}
\mc{A}_1 = \mc{A}*_N\mc{B}*_L\mc{B}^\dag_{\mc{N},\mc{P}} = \mc{A}.
\end{equation*}
Hence, $ (\mc{A}*_N\mc{B})^\dag_{\mc{M},\mc{N}} = \mc{B}^\dag_{\mc{N},\mc{P}}*_N\mc{A}^\dag_{\mc{M},\mc{N}} $.
\end{proof}

We next present another characterization of the product of arbitrary-order tensors,
as follows,

\begin{theorem}
Let $\mc{A} \in \mathbb{C}^{I_1\times\cdots\times I_M \times J_1
\times\cdots\times J_N}$ and $ \mc{B} \in \mathbb{C}^{J_1\times\cdots\times J_N \times K_1 \times\cdots\times K_L}$.\\
Let $ \mc{M} \in \mathbb{C}^{I_1\times\cdots\times I_M \times I_1 \times\cdots\times I_M} $, $\mc{N} \in \mathbb{C}^{K_1\times\cdots\times K_L \times K_1 \times\cdots\times K_L} $  and $ \mc{P} \in \mathbb{C}^{J_1\times\cdots\times J_N \times J_1 \times\cdots\times J_N} $
are  three Hermitian positive definite tensors. Then 
\begin{eqnarray*}
 (\mc{A}*_N\mc{B})_{\mc{M},\mc{N}}^{\dg} = (\mc{B}_1)^\dg_{\mc{\mc{P},\mc{N}}}\n (\mc{A}_1)^\dg_{\mc{M},\mc{P}},
\end{eqnarray*}
where $ \mc{A}_1 = \mc{A}\n\mc{B}\1 \mc{B}^\dg_{\mc{P},\mc{I}_L} $ and $ \mc{B}_1 = (\mc{A}_1)^\dg_{\mc{M},\mc{P}}\m \mc{A}_1 \n \mc{B}.$  
\end{theorem}

\begin{proof}
Let $ \mc{X} = \mc{A}\n \mc{B} $ and $ \mc{Y} = (\mc{B}_1)^\dg_{\mc{P},\mc{N}}\n (\mc{A}_1)^\dg_{\mc{M},\mc{P}}$. Now we have
\begin{equation}\label{4.5}
    \mc{A}\n \mc{B} = \mc{A}\n \mc{B}\1 \mc{B}^\dg_{\mc{P},\mc{I}_L}\n \mc{B} = \mc{A}_1\n \mc{B} = \mc{A}_1 \n (\mc{A}_1)^\dg_{\mc{M},\mc{P}}\m \mc{A}_1 \n \mc{B} = \mc{A}_1 \n \mc{B}_1.
\end{equation}

Now, using Eq. \eqref{4.5}, we obtain
\begin{eqnarray}\label{451}
 \mc{X}\1\mc{Y} \m \mc{X} 
&=& \mc{A}_1\n\mc{B}_1\1(\mc{B}_1)^\dg_{\mc{P},\mc{N}} \n \mc{B}_1
= \mc{X},\\\label{452}
\mc{Y}\m\mc{X}\1 \mc{Y} 
&=&  (\mc{B}_1)^\dg_{\mc{P},\mc{N}}  \n \mc{B}_1 \1 (\mc{B}_1)^\dg_{\mc{P},\mc{N}}\n (\mc{A}_1)^\dg_{\mc{M},\mc{P}}
= \mc{Y},\\\label{453}
\mc{N}\1\mc{Y}\m\mc{X}
&=& \mc{N}\1 (\mc{B}_1)^\dg_{\mc{P},\mc{N}} \n \mc{B}_1 = (\mc{N}\1\mc{Y}\m\mc{X})^*.
\end{eqnarray}

Further, using the following relations
\begin{equation*}
  \mc{B}_1 \1 \mc{B}^\dg_{\mc{P},\mc{I}_L} =   (\mc{A}_1)^\dg_{\mc{M},\mc{P}}\m \mc{A}_1 ~~and ~~\mc{B}_1 \1 \mc{B}^\dg_{\mc{P},\mc{I}_L} = \mc{B}_1\1(\mc{B}_1)^\dg_{\mc{P},\mc{N}}\n (\mc{A}_1)^\dg_{\mc{M},\mc{P}}\m \mc{A}_1,
\end{equation*}
we have
\begin{equation*}
(\mc{A}_1)^\dg_{\mc{M},\mc{P}}\m \mc{A}_1 = \mc{B}_1\1(\mc{B}_1)^\dg_{\mc{P},\mc{N}}.    
\end{equation*}

It concludes that,
 \begin{equation}\label{454}
 \mc{M}\m \mc{X}\1\mc{Y} = \mc{M}\m \mc{A}_1 \n (\mc{A}_1)^\dg_{\mc{M},\mc{P}}  
 = (\mc{M}\m \mc{X}\1\mc{Y})^*.
 \end{equation}
From the relations \eqref{451}-\eqref{454} validates $ \mc{Y}= \mc{X}^\dg_{\mc{M},\mc{N}}$. Hence $(\mc{A}*_N\mc{B})_{\mc{M},\mc{N}}^{\dg} = (\mc{B}_1)^\dg_{\mc{\mc{P},\mc{N}}}\n (\mc{A}_1)^\dg_{\mc{M},\mc{P}}.$
This completes the proof.
\end{proof}

The significance of the properties of range and null space of arbitrary-order tensors, the last result achieved the sufficient condition for the triple reverse order law of tensor. 

\begin{theorem}\label{uvw2}
Let $ \mc{U}\in \mathbb{C}^{I_1\times\cdots\times I_M \times J_1
\times\cdots\times J_N} $, $ \mc{V}\in \mathbb{C}^{J_1\times\cdots\times J_N \times K_1
\times\cdots\times K_L} $ and \\
$ \mc{W}\in \mathbb{C}^{K_1\times\cdots\times K_L \times H_1
\times\cdots\times H_R} $. Let $\mc{M}  \in \mathbb{C}^{I_1\times\cdots\times I_M \times I_1 \times\cdots\times I_M} $ and $\mc{N} \in \mathbb{C}^{H_1\times\cdots\times H_R \times H_1 \times\cdots\times H_R}$ be a pair of Hermitian positive definite tensors. If 
\begin{equation*}
\mathfrak{R}(\mc{W}) \subseteq \mathfrak{R}[(\mc{U} \n \mc{V})^*] ~~ and~~  \mathfrak{R}(\mc{U}^*) \subseteq \mathfrak{R}(\mc{V}\1\mc{W}).
\end{equation*}
Then 
\begin{equation*}
 (\mc{U}\n\mc{V}\1\mc{W})^\dag_{\mc{M},\mc{N}} = \mc{W}^\dg_{\mc{I}_L,\mc{N}}\1\mc{V}^\dg\n \mc{U}^\dg_{\mc{M},\mc{I}_N}.  
\end{equation*}

\end{theorem}

\begin{proof}
 Let $ \mc{A} = \mc{U}\n\mc{V}\1\mc{W}, ~~  \mc{W}_1 = (\mc{U}\n \mc{V})^\dg \m \mc{A} $  and $ \mc{U}_1 = \mc{A} \f (\mc{V}\1 \mc{W})^\dag $.
 From Eq.\eqref{lemma37}, we get
 \begin{eqnarray*}
 \mathfrak{R}(\mc{U}_1) 
 = \mathfrak{R}(\mc{A}) 
~~ \textnormal{and}~~ \mathfrak{R}(\mc{W}_1^*) 
= \mathfrak{R}(\mc{A}^*).
 \end{eqnarray*}
 Also from Lemma \ref{ID}(c), we get  
 \begin{eqnarray*}
 &&\mathfrak{R}[(\mc{W}_1^\dag)^*] =\mathfrak{R}(\mc{W}_1) \subseteq \mathfrak{R}[(\mc{U}*_N\mc{V})^\dag]  =\mathfrak{R}[(\mc{U}*_N\mc{V})^*]\\
 &&\hspace{3cm}\textnormal{and}~~ \mathfrak{R}(\mc{U}_1^\dag) = \mathfrak{R}(\mc{U}_1^*) \subseteq \mathfrak{R}[(\mc{V}\1\mc{W})^\dag]^* = \mathfrak{R}(\mc{V}\1\mc{W}).
 \end{eqnarray*}

Applying Lemma~\ref{ID}[(a),(b)] and Lemma~\ref{IDR1}[(a),(b)], we have
\begin{eqnarray*}
\mc{W}_1^\dag *_L\mc{V}^\dag *_N\mc{U}_1^\dag 
&=&  \mc{W}_1^\dag *_L\mc{W}_1*_R(\mc{V}*_L\mc{W})^\dag *_N\mc{U}_1^\dag 
= \mc{A}^\dg \m \mc{U}_1 \n \mc{U}_1^\dg 
=  \mc{A}^\dg,
\end{eqnarray*}
which is equivalent to
 \begin{equation}\label{uvw2.1}
   (\mc{U}\n\mc{V}\1 \mc{W})^\dg = [(\mc{U}\n\mc{V})^\dg \m \mc{A}]^\dg \1 \mc{V}^\dg \n [\mc{A}\f (\mc{V}\1\mc{W})^\dg]^\dg. 
\end{equation}
Replacing $ \mc{U}$ and $\mc{W} $ by $ \mc{M}^{1/2}\m\mc{U} $ and $\mc{W}\f \mc{N}^{-1/2} $ in Eq.\eqref{uvw2.1} along with  using Eq.\eqref{2.17} and Lemma~\ref{42}[(a),(b)] we have
\begin{eqnarray*}
 \mc{A}^\dg_{\mc{M},\mc{N}} 
&=& \mc{N}^{-1/2}\f [(\mc{M}^{1/2}\m\mc{U}\n \mc{V})^\dg\m \mc{M}^{1/2}\m \mc{A}\f \mc{N}^{-1/2}]^\dg \1 \mc{V}^\dg*_N \\
&& \hspace{3cm}
 [\mc{M}^{1/2}\m \mc{A}\f \mc{N}^{-1/2}\f (\mc{V}\1\mc{W}\f\mc{N}^{-1/2})]^\dg \m \mc{M}^{1/2} \\
&=& \mc{N}^{-1/2}\f [\mc{W}_1\f\mc{N}^{-1/2}]^\dg \1 \mc{V}^\dg \n [\mc{M}^{1/2}\m \mc{U}_1]^\dg \m \mc{M}^{1/2},\\
&=& (\mc{W}_1)^\dg_{\mc{I}_L,\mc{N}}\1\mc{V}^\dg \n (\mc{U}_1)^\dg_{\mc{M},\mc{I}_N}.
\end{eqnarray*}
Applying Lemma \ref{ID}[(a),(c)] and Lemma \ref{IDR1}(a) in the given condition, we get
\begin{eqnarray*}
\mc{W} = (\mc{U}\n\mc{V})^\dg \m (\mc{U}\n\mc{V})\1\mc{W} = \mc{W}_1\textnormal{~~ and~~} \mc{U} = \mc{U}\n (\mc{V}\1\mc{W})\f(\mc{V}\1\mc{W})^\dg = \mc{U}_1.
\end{eqnarray*}
Hence, 
$(\mc{U}\n\mc{V}\1\mc{W})^\dg_{\mc{M},\mc{N}} = \mc{W}^\dg_{\mc{I}_L,\mc{N}}\1\mc{V}^\dg\n \mc{U}^\dg_{\mc{M},\mc{I}_N}.$
\end{proof}

\section{Conclusion}
In this paper, a novel SVD and full rank-decomposition of arbitrary-order tensors using reshape operation is developed.  Using this decomposition, we have studied the Moore-Penrose and general weighted Moore-Penrose inverse for arbitrary-order tensors via the Einstein product.  We have also added some results on the range and null spaces to the existing theory.  Then we discuss a few characterizations of cancellation properties for Moore-Penrose inverse of tensors. In addition to these, we have discussed the reverse-order laws for weighted Moore-Penrose inverses. 
In the future, it will be more interesting to express additional identities of weighted Moore-Penrose inverse in terms of the ordinary Moore-Penrose inverse for arbitrary-order tensor.

{\bf\large{Acknowledgments}}\\
The first and the third authors are grateful to the Mohapatra Family Foundation and the College of Graduate Studies, University of Central Florida, Orlando, for their financial support for this research. In addition to this the first author is grateful for the partially supported by Science and Engineering Research Board (SERB), Department of Science and Technology, India, under the Grant No. EEQ/2017/000747.

\bibliographystyle{amsplain}
\bibliography{WMPIUCF}
\end{document}